\documentclass{elsarticle}
\usepackage{geometry}                		
\usepackage{graphicx}
\usepackage{amsmath,amsthm, amssymb}
\usepackage{mathtools}
\usepackage{enumerate}
\usepackage{enumitem}
\usepackage{amsfonts}
\usepackage[utf8]{inputenc}
\usepackage{latexsym,lineno}
\usepackage{listings}
\usepackage{hyperref}
\usepackage{wrapfig}
\usepackage[outline]{contour}
\usepackage{cleveref}
\usepackage{tikz}
\usepackage{tikz-cd}
\usepackage{tikzit}
\usepackage{caption}
\usepackage{centernot}
\usepackage{stmaryrd}
\usepackage{wasysym}

\tikzstyle{Vertex}=[fill={rgb,255: red,142; green,142; blue,142}, draw=black, shape=rectangle]
\tikzstyle{white vertex}=[fill=white, draw=black, shape=circle]
\tikzstyle{grey vertex}=[fill={rgb,255: red,142; green,142; blue,142}, draw=black, shape=circle]
\tikzstyle{black vertex}=[fill=black, draw=black, shape=circle]
\tikzstyle{small black}=[fill=black, draw=black, shape=circle, scale=.5pt, tikzit draw=white]
\tikzstyle{black empty}=[fill=white, draw=black, shape=circle]
\tikzstyle{small black empty}=[fill=white, draw=black, shape=circle, scale=.5pt, tikzit draw={rgb,255: red,156; green,156; blue,156}, tikzit fill=white]

\tikzstyle{double blue edge}=[->, draw={rgb,255: red,6; green,118; blue,255}, latex-latex]
\tikzstyle{automorphism}=[-{to}]
\tikzstyle{new edge style 0}=[->, draw={rgb,255: red,6; green,118; blue,255}, -latex]
\tikzstyle{Edge}=[-]
\tikzstyle{Arc}=[->, -latex]
\tikzstyle{blue_edge}=[-, draw={rgb,255: red,0; green,146; blue,146}, line width=1.2pt]
\tikzstyle{edge red}=[-, fill=none, draw=red]
\tikzstyle{green edge}=[-, draw={rgb,255: red,0; green,168; blue,0}, -latex]
\tikzstyle{big arc}=[->, line width=1.2pt, -latex]

\usetikzlibrary{decorations.markings}
\tikzset{negated/.style={
        decoration={markings, mark= at position 0.5 with {\node[transform shape] (tempnode) {$\backslash$};}}
                ,postaction={decorate}}}

\newcommand{\Ra}{\Rightarrow}
\newcommand{\La}{\Leftarrow}
\newcommand{\ra}{\rightarrow}
\newcommand{\la}{\leftarrow}
\newcommand{\rra}{\twoheadrightarrow}
\newcommand{\AND}{\text{ and }}
\newcommand{\OR}{\text{ or }}
\newcommand{\Sg}{\Sigma}
\newcommand{\tx}[1]{\text{#1}}
\newcommand{\set}[1]{\{#1\}}
\newcommand{\core}[1]{\overbracket[0.3pt][0pt]{\overbracket[0.3pt][0pt]{\mkern-2.9mu #1}}}
\renewcommand{\cal}[1]{ \mathcal{#1} }
\newcommand{\up}{\nearrow}
\newcommand{\dw}{\searrow}
\newcommand{\F}{\cal{F}}

 \newcommand{\tikzLarrow}{%
 	\tikz[->]{
 		\draw[-{Triangle[angle=60:4.5pt]}, line width=0.6pt] 
 		(1em,0) -- (0,0);
 	}%
 }

\newcommand{\tikzRarrow}{%
  \tikz[<-]{
    \draw[-{Triangle[angle=60:4.5pt]}, line width=0.6pt] 
      (0,0) -- (1em,0);
  }%
 }

\tikzset{>=latex}

\newcommand{\tikzdigon}{%
\tikz{%
    \draw[{Triangle[angle=60:4.5pt]}-{Triangle[angle=60:4.5pt]}, line width=0.6pt] 
      (0,0) -- (1.2em,0);
  }%
}%

\newcommand{\nottikzRarrow}{%
  \tikz{%
    \draw[-{Triangle[angle=60:4.5pt]}, line width=0.6pt] (0,0) -- (1em,0);
    \draw[line width=0.5pt] (0.3em,-0.3em) -- (0.7em,0.3em);
  }%
}

\newcommand{\rvec}[1]{\reflectbox{\ensuremath{\vec{\reflectbox{\ensuremath{#1}}}}}}

\newcommand{\digon}{\tikzdigon}

\newcommand{\cone}{{\!}^{\triangleleft}}
\newcommand{\mino}[1]{\left|#1\right|_{_{min}}}

\newcommand{\apex}[1]{\star_{_{#1}}}

\newtheorem{theorem}{Theorem}
\newtheorem{lemma}[theorem]{Lemma}
\newtheorem{proposition}[theorem]{Proposition}
\newtheorem{corollary}[theorem]{Corollary}
\newtheorem{definition}[theorem]{Definition}
\newtheorem{example}[theorem]{Example}
\newtheorem*{observation}{Observation}

\date{}

\begin{document}

\makeatletter
\def\ps@pprintTitle{%
	\let\@oddhead\@empty
	\let\@evenhead\@empty
	\def\@oddfoot{\reset@font\hfil\thepage\hfil}
	\let\@evenfoot\@oddfoot
}
\makeatother
	
\begin{frontmatter}
	
\title{On core of categorical product of (di)graphs}
\author{Reza Naserasr} 
\author{Cyril Pujol} 

\address{Universit\'e Paris Cit\'{e}, CNRS, IRIF, F-75013, Paris, France.}
\address{Emails:\text{\{reza,cpujol\}}@irif.fr}
\begin{abstract}
	The core of a graph is the smallest graph (in terms of number of vertices) to which it is homomorphically equivalent.
	The question of the possible order of the core of the tensor product (also known as categorical, Heidetnemi or direct product) of two graphs captures some well known problems. For instance, the recent counterexample to the Hedetniemi conjecture for 5-chromatic graphs is equivalent to saying that there are cores of order at least 5 whose product has a core of order 4.
	
	In this work, motivated by a question from Leonid Libkin in the area of graph databases, we first present methods of building cores whose categorical product is also a core. Extending on this we present sufficient conditions for a set of cores to have a product which is also a core. Presenting an example of such a family of digraphs, we construct a family of $\binom{2n}{n}$ digraphs, where the number of vertices of each is between $n^2+5n+2$ and $3n^2+3n+2$ and the product is a core. We then present a method of transforming the example into a family of graphs.
\end{abstract}


\end{frontmatter}

	\section{Introduction}	
	\label{sec:intro}
		In the categories of combinatorial structures such as automatas, graphs, relational structures, a natural question is that of the growth factor of the product operator. 
		In that regard the category of (di)graphs plays a central role because it is one of the simplest of all such structures and yet captures most of the complexity of the study.
		Let us first recall the basic definitions. Given (di)graphs $G$ and $H$, a \emph{homomorphism} of $G$ to $H$ is a mapping of $V(G)$ to $V(H)$ which preserve adjacencies and in the case of digraphs also the direction of the edges. When there is a homomorphism of $G$ to $H$ we write $G\to H$. Given a class $\mathcal{C}$ of graphs, if every member $G$ of $\mathcal{C}$ satisfies $G\to H$, then we may write $\mathcal{C}\to H$. If $G\to H$ and $H \to G$, then we say $G$ and $H$ are \emph{homomorphically equivalent}, or just equivalent, and write $G \simeq H $. The category of (di)graphs is the category with objects being (di)graphs and \emph{morphisms} being the homomorphisms.		
		
		If $G$ admits a homomorphism to $H$ and every such homomorphism is surjective, then we write $ G \rra H$. Thus $G\to K_k$ is the same as saying that $G$ is $k$-colorable while $G \rra K_k$ is to claim that $G$ is $k$-chromatic.

		The \emph{core} of a (di)graph $G$ , denoted $\core G$, is defined to be smallest (di)graph verifying $G \simeq \core G$. It is not hard to verify that $\core G$ is always a subgraph of $G$ and is unique up to isomorphism, we refer to \cite{HN2004graphs} for a proof and for more details on these subjects.
		
		A \emph{retract} of a graph $G$ is an endomorphism which is also a projection. In other words, $\phi: G \ra G$ verifies $\phi^2 = \phi$. A useful caracterisation of cores is: A (di)graph is a core if its only retract is the identity.
		
		Following the general notion of category, the \emph{categorical product} of two (di)graphs $G$ and $H$ is a (di)graph, denoted $G\times H$, which maps to both $G$ and $H$ and has the property that if a graph $F$ maps to both $G$ and $H$, then it also maps to $G\times H$. Obviously all the graphs equivalent to $G\times H$ work for this definitions. The most typical member of this equivalent class is the following.

		\begin{definition}[Tensor product]
			The categorical or tensor product of two (di)graphs $G$ and $H$ is the (di)graph $G \times H$ whose vertex set is $V_G \times V_H$ where a pair of vertices $(u_1,v_1)$  and $(u_2,v_2)$ forms an edge (or arc) whenever $u_1u_2$ is an edge (arc) of $G$ and $v_1v_2$ is an edge (arc) of $H$. 
		\end{definition}
	
	One of the first observations in the study of category of graphs is that $G\times H$ is indeed the categorical product of (di)graphs. More precisely the diagram in \Cref{fig:tensor_product} commutes. For details we refer to \cite{HN2004graphs}.

	From a category standpoint each graph $G$ is representing the class of graphs which are homomorphically equivalent to it. As the smallest member of this class (in terms of number of vertices) is $\core G$, the core of $G$, it is natural to work with cores. Given a graph (or product of graphs) then one natural question is the order of its core. 
	
	This innocent looking question, surprisingly, captures some well known questions and results. For example observing that the only core on four vertices is the complete graph $K_4$, the recent result of Tardif on Hedetniemi conjecture can be stated as:
	
	\begin{theorem}[Tardif, \cite{T23}]
		There are graphs $G$ and $H$ whose cores are not of order 4, but $\core{G\times H}$ is of order 4.
	\end{theorem}
	
	To see that such $G$ and $H$ must be counterexamples to the Hedetniemi's conjecture observe that if one of $G$ or $H$ is 3-colorable, then the product is 3-colorable and thus its core cannot be of order 4. Next suppose both $G$ and $H$ are of chromatic number at least 4 and that at least one, say $G$ is of chromatic number 4. If the core of $G\times H$ is of order 4, then it must contain $K_4$, and thus both $G$ and $H$ must contain $K_4$. As $G$ is 4-chromatic and contains $K_4$ its core must be $K_4$, but that contradicts the assumption on $G$. If $G$ and $H$ both have chromatic number at least 5, but the core of product is $K_4$, then we have two 5-chromatic graphs whose product is 4-chromatic, and a counterexample to Hedetniemi's conjecture. We note that first counterexample to this conjecture was build by Shitove~\cite{S19}. After a series of improvement, Tardif\cite{T23} built a counterexample for chromatic number 5.  
	
    The motivation for our work however comes from the database theory. 
    
    \subsection{Categorical product in database theory}
    
    In database theory one way to view recorded data is to see them as an array, in other words relational structure. There are then two types of entries: constants which are fixed or nulls and variables whose value can be set by the user. To search for a specific record (formally a query) then is to find (or count the number of) homomorphisms from this record to the given database. For more on this we refer to \cite{GLS14}.

    The relation between queries in database theory and categorical product was first studied in \cite{AL18}. It is observed that the certain answer to a relational query on an incomplete database $D$ is determined by the (core of) of the categorical product of all possible completions of $D$. The possible order of the core then has a theoretical impact on the efficiency of potential algorithms. It is shown in \cite{AL18} that for relational queries defined in Monadic Second-Order (MSO) logic there are families of databases of size polynomial of $n$ whose core of the product is of size double exponential in order of $n$. For relational queries defined by First-Order (FO) logic, an exponential lower bound was already observed in \cite{ABKR17}. This is based on the tensor product of digraphs; by taking all directed cycles of order $p$ where $p$ is any prime number smaller than $n$, we get a digraph whose core is the directed cycle of length equal to the product of all these prime numbers, and thus, by the prime number theorem, its order is exponential in terms of $n$.  Then in \cite{AL18} it is left as a ``very hard'' problem to decide if the order of the core of product of a set of incomplete databases, each of order $o(n)$, is always bounded by an exponential function of $n$.
    
    In this work, towards addressing this question, in \Cref{sec:BuildingProduct}, we first present sufficient conditions for the product of two cores to be a core and then provide a sufficient condition for the product of a family of cores to be a core. To complete this process, in \Cref{sec:Decreasing_mountains}, we present a family of $\binom{2k}{k}$ digraphs each of order between $k^2+3k+2$ and $k^2+5k+2$ whose product is a core. In \Cref{sec:Digraph-Graph} we present a gadget using which we can transform our digraphs to graphs having the same property. We settle our notation and terminology in the following section. The notion of orthogonality, introduced in \Cref{sec:Extended_Paths}, seems to be of its own interest, as it captures the notion of a graph being multiplicative.
    		
\section{Notation and terminologies}
	A glossary of all symbols used in this work is provided in \Cref{tab:symbols}. We use the terms digraphs, oriented graphs, and graphs, noting the first is a binary (ordered) relation on a set called vertices where the relations are referred to as arcs. The second (oriented graph) is any antisymmetric digraph, that is to say if $u\tikzRarrow v$, then $v \nottikzRarrow u$. Finally the last term, i.e., graph, is a symmetric digraph, that is to say if $u\tikzRarrow v$, then $u \tikzRarrow v$. Loops and parallel edges are not considered in this work.  A statement claimed on ``(di)graphs'' holds on all three classes. 
	
	\begin{definition}
		A \emph{homomorphism} of a (di)graph $G$ to a (di)graph $H$ is a mapping $\phi: V(G) \ra V(H)$ verifying $u \tikzRarrow v \Ra \phi(u) \tikzRarrow \phi(v)$.  When there is such a homomorphism we write $G \ra H$.  When $G \ra H$ and $H \ra G$ both hold we say $G$ and $H$ are hom-equivalent and write $G \simeq H$.
	\end{definition}
	
	Homomorphisms are the \emph{morphisms} of the category of graphs. The relation $G\ra H$ is a quasi order which become a partial order when restricted to equivalent classes of graphs, or on the class of cores. 
	Noting that number of edges are of little importance in this work, and that ``order'' is used for homomorphism order, in this work \emph{size} of a graph $G$ is the number of its vertices, denoted $|G|$. 				
		\begin{definition}[retract]
			A retract is a surjective homomorphism of $G$ to a subgraph $H$ such that it is an identity when restricted on $V(H)$. That is a mapping $\phi: G \ra G$ which verifies $\phi^2 = \phi$.
		\end{definition}
		
		A characterization of cores using the notion of retract is as follows.
		
		\begin{proposition}
			A (di)graph is a core if and only its only retract is the identity. 
		\end{proposition} 
		In particular, if $G$ is a core, then the only vertex dominating a vertex $u$ is $u$ itself.
        The key operation studied in this work is the following. 
		
		\begin{definition}[Tensor product]
			The tensor product of two (di)graphs $G$ and $H$, denoted $G\times H$, is the graph on  $V(G) \times V(H)$ with the edge set $\set{(u_1,v_1) \tikzRarrow (u_2,v_2) | u_1 \tikzRarrow u_2 \AND v_1 \tikzRarrow v_2}$. 
		\end{definition}
		
		This is the classic notion of product in the category of graphs. Moreover, $G\times H$ is the joint of $G$ and $H$ in the homomorphism order. The product $G\times H$ is known under various names in the literature, for example is often called the categorical product, the direct product or the Hedetniemi product. We refer to \cite{HN2004graphs} for more on the subject.
		
		The fundamental property of the tensor product is:
		\begin{proposition}[Fundamental property of the product]
			Let $X,G,H$ be (di)graphs. Then ${X \ra G \times H} \iff {X \ra G \AND X \ra H}$
		\end{proposition}
		
		In fact it is true, moreover, that the diagram of \Cref{fig:tensor_product} commutes for every choice of $\phi_G$,$\phi_H$. In this diagram, $\pi_G$ and $\pi_H$ are the projections of $G\times H$ to $G$ and $H$ respectively. 
		
		
		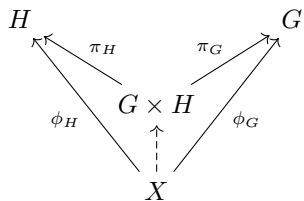
\begin{figure}[htb]
		\centering
			\begin{tikzcd}
				H &                                             & G \\
 				& G\times H \arrow[lu,"\pi_H"'] \arrow[ru,"\pi_G"]             &   \\
  				& X \arrow[ruu,"\phi_G"'] \arrow[luu, "\phi_H"] \arrow[u,dashed] &  
				\end{tikzcd}
		\caption{Universal property of the tensor product}
		\label{fig:tensor_product}
		\end{figure}
		
		Some elementary properties of the core of product are as follows: 
		\begin{proposition}
			If $G,H$ are (di)graphs, 
			
			\begin{itemize}
				\item $\core{G \times G} = \core G$
				\item $G \ra H \Ra \core{G \times H} = \core G $
				\item $ |G \times H| = |G| \times |H|$
				\item $G \rra \core G$
			\end{itemize}
		\end{proposition}
		
		A less trivial but still well-known fact is the following.
		
		\begin{proposition}
			Given integers $n_1,n_2, \dots, n_k$ the core of the product of the directed cycles $C_{n_1},C_{n_2}, \dots, C_{n_k},$ is the directed cycle $C_r$ where $r=lcm(n_1,n_2, \dots, n_k)$ is the least common multiple of $n_1,n_2, \dots, \text{ and } n_k$.
		\end{proposition}
		
		Combined with the prime number theorem then this is enough to show the existence of oriented graphs of order at most $n$ whose core of product has size exponential in terms of $n$.

	\begin{table}[htb]
			\centering
			\begin{tabular}{|c|c|c|}
				\hline
				Symbol & Meaning & Definition \\
				\hline
				\hline
				
				$\vec P_k$ & directed path of length $k-1$ &\\
				\hline
				$u \tikzRarrow v$ & $v$ is an out-neighbor of $u$ &\\ 
				\hline
				$u \tikzdigon v$ & $u \tikzRarrow v$ and $u \tikzLarrow v$ &\\ 
				\hline
				$G \ra H$ & there exists a hom. of $G$ to $H$  &\\ 
				\hline
				$H \la G$ & there exists a hom.  of $G$ to $H$  &\\ 
				\hline
				$G \not \ra H$ & No hom. of $G$ to $H$ & \\
				\hline
				$G \simeq H$ & $G \ra H$ and $G \la H$  &\\ 
				\hline
				$G \cong H$ & $G \AND H$ are isomorphic  &\\ 
				\hline
				$G\rra H$ &$G \ra H$ and any hom. is surjective& \\
				\hline
				$G\stackrel{incl.}{\hookrightarrow} H$ & inclusion of $G$ in $H$ & \\
				\hline
				$|G|$ & order of $G$& \\
				\hline 	\\[-1.1em]
				$\core G$ & core of $G$& \Cref{sec:intro}\\
				\hline
				$\pi_G$ &  projection of $G\times H$ to $G$ & \Cref{fig:tensor_product}\\
				\hline
				$G \perp H$ &  $G \times H \rra G \; $ and $ \; G \times H \rra H$ & \Cref{def:orthogonal}\\
				\hline
				$h(P)$ &  height of the $k$-b-path $P$ & \Cref{subsec:oriented_paths}\\
				\hline
				$P \vee Q$ & sum of the $k$-b-paths $P \AND Q$ & \Cref{def:sum_paths}\\
				\hline
				$\mino{P}$ &  order of the smallest component of $P$ & \Cref{def:mino}\\
				\hline
				$P \wedge Q$ &  product of the $k$-b-paths $P \AND Q$& \Cref{def:path_prod}\\
				\hline
				$\cone G$, $\apex G$ &  cone of $G$ and its apex & \Cref{def:cone}\\
				\hline
				$G \Bowtie H$ &  complete join of $G$ and $H$ & \Cref{sec:Digraph-Graph}\\
				\hline
			\end{tabular}
			\caption{Glossary of symbols}
			\label{tab:symbols}
		\end{table}
	
	\section{Orthogonality and algebra of Paths}\label{sec:Extended_Paths}

	In this section we first introduce the concept of orthogonality for graphs. We show how this concept, combined with the notion of orthogonality invariant functions, captures some classic notions such as the notion of multiplicative graphs. We then study the algebra of homomorphism on a certain family of oriented graphs.

	\subsection{Orthogonal (di)graphs}
	Toward our main goal, a first thing is to construct a large family $\F$ of digraphs verifying 
		$$\forall \,  G,H \in \F, \core{G\times H} \rra \core{G} \AND \core{G\times H} \rra \core{H}$$
	 
	The following property is key in building such families:
	
	\begin{proposition}\label{prop:CoreProducNotSurjective}
		For each pair $G$ and $H$ of (di)graphs, there are two subgraphs $G'$ and $H'$ such that first of all $G\times H \simeq G'\times H'$, secondly, any mapping of $\core{G\times H}$ to either $G'$ or $H'$ is surjective.
		
	\end{proposition}
	
    \begin{proof}
		Since $G\times H$ maps to both $G$ and $H$, there are minimal subgraphs $G'$ and $H'$ of $G$ and $H$, respectively, to which $G\times H$ maps. The claim can be checked readily now.		
	\end{proof}
	
	A different way to state this proposition is that given any pair $G,H$ of (di)graphs, one can delete some vertices of $G$ and $H$ without changing the core of the product but obtaining the property that the core of the product now maps only surjectively to both of the subgraphs obtained. The essence of the claim is presented in diagram of Figure~\ref{fig:DiagramP'Q'} where the arrows that are implied from transitivity are not drawn.


	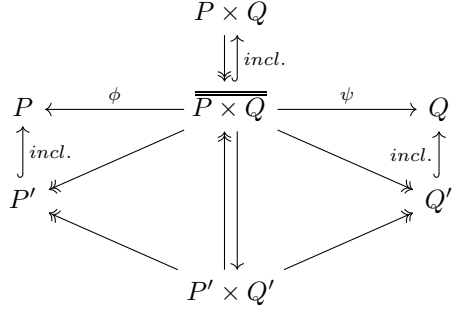
\begin{figure}
		
\centering
\begin{tikzcd}
                             &  & P \times Q \arrow[d, two heads, shift right]                                                                                                                                           &  &                            \\
P                            &  & \core{P \times Q} \arrow[u, "incl."', hook, shift right] \arrow[ll, "\phi"'] \arrow[rr, "\psi"] \arrow[lld, two heads] \arrow[rrd, two heads] \arrow[dd, shift left] &  & Q                          \\
P' \arrow[u, "incl."', hook] &  &                                                                                                                                                                                        &  & Q' \arrow[u, "incl.", hook] \\
                             &  & P' \times Q' \arrow[llu, two heads] \arrow[rru, two heads] \arrow[uu, two heads, shift left]                                                                                           &  &                           
\end{tikzcd}
	
	\caption{Diagram of Proposition~\ref{prop:CoreProducNotSurjective}}
	\label{fig:DiagramP'Q'}   
	\end{figure}
	
	\begin{definition}
	\label{def:orthogonal}
		We say that two (di)graphs $G$, and $H$ are \emph{orthogonal} if any mapping of $G\times H$ to either of $G$ and $H$ is surjective. In such a case, we write $G \perp H$.	
	\end{definition}
	 In other words $G$, and $H$ are orthogonal if the subgraphs $G'$ and $H'$ of the proposition above when applied for the pair $G$ and $H$ are $G$ and $H$ themselves. Note that if $G \perp H$, then both $G$ and $H$ must be cores and, in particular, $G \perp G$ if and only if $G$ is a core.

	\begin{proposition}
		If $G \perp H$, then $|\core{G\times H}| \geq \max(|G|, |H|)$
	\end{proposition}

	\begin{proposition}
		\label{prop:project+hom}
		If $G \perp H$, and $\phi: G\times H \ra G \times H$ is a homomorphism, then $\forall x \in V(H), \exists u \in V(G)$ such that $ (u,x)\in Im(\phi)$ 
	\end{proposition}
	\begin{proof}
		By definition of orthogonality, the mapping $\pi_H \circ \phi$ must be surjective.
	\end{proof}
	A family  $\F$ of (di)graphs is said to be an \emph{orthogonal family} if $\prod\limits_{H\in \cal F}H \rra G$ for any $G\in \F$. One way of building an orthogonal family is by using orthogonality invariant as defined below.
	
	\begin{definition}
	\label{def:tau}
	A mapping $\tau$ of (di)graphs to a set $X$ is said to be an orthogonality invariant if the following two condition holds:
	\begin{enumerate}
		\item For any pair $G,H$ of (di)graphs $ \tau(G) = \tau(H) = c \Ra  \tau(G\times H) = c$
		\item For any three (di)graphs $F,G,H$, if $ F \ra G \ra H  \AND c = \tau(F) = \tau(H)\, $ then $ \tau(G) = c$.
	\end{enumerate}
	
	Given an orthogonality invariant $\tau$ and and element $c$ of $X$ we define $M_{\tau}(c)$ to be the minimal elements (with respect to subgraph order) of $\tau^{-1}(c)$.	
	\end{definition}

	The key point of the definition is the following.
	
	\begin{proposition}[Orthogonality invariant]
	Given an orthogonality invariant $\tau$ with $X$ as the range, and an element $c$ of $X$, the set $M_{\tau}(c)$ is an orthogonal family. 
	\end{proposition}
	
	\begin{proof}
	We need to show that any mapping, say $\phi$, of $\prod\limits_{H\in \cal F}H$ to $G$ is surjective for any $G$ in $\cal F$. Observe that $\tau\left (\prod_{H\in \cal F}H \right ) = c$ (by \Cref{def:tau}(1)). 
	As $\prod\limits_{H\in \cal F}H \ra \phi \left (\prod\limits_{H\in \cal F}H \right ) \stackrel{incl.}{\hookrightarrow} G$, we have $ \tau  \left ( \phi \left (\prod\limits_{H\in \cal F}H \right ) \right ) = c$ . And since $G$ is minimal, then $Im(\phi) = G$.
    \end{proof}

	Towards a better understanding we present a few examples. 
    
    \begin{example}
    	Let $X=\{T, F\}$ and $\tau$ be a mapping of digraphs to $X$ which assigns $T$ to the digraph $G$ if the directed path of length 5 maps to $G$, and otherwise assigns F to $G$. Since there is a directed $5$-path in the product of two 5-paths, the first condition of being an orthogonality invariant holds. The second condition simply follows by transitivity. Then $M_{\tau}(T)$ consists of: the directed path of length 5 and directed cycles of lengths 2,3, and 4. It can also be directly verified that any pair of these digraphs are orthogonal. On the other hand, $M_{\tau}(F)$ consists of a single digraph which is on one vertex and has no arc. 
    	
    \end{example}
    
    	If we replace the directed path of length $5$ with the directed path of length $k$ in this example, then $M_{\tau}(T)$ will consist of the directed path of length $k$, and all directed cycles of lengths at most $k-1$. 	
    	
    \begin{example}
    	By taking $X=\{T, F\}$ again, and considering the class of graphs, we define $\tau$ by 3-colorability: if $\chi(G)\leq 3$, then $\tau(G)=T$ and if $\chi(G)\geq 4$, then $\tau(G)=F$. That $\tau$ is an orthogonality invariant follows from the well know result of El-Zahar and Sauer \cite{ES85}, which is the only nontrivial and verified case of the Hedetniemi conjecture.
    	
    	More precisely, if $\tau$ assigns $F$ to two graphs, meaning neither is 3-colorable, then, by the result of \cite{ES85}, their product is also not 3-colorable. If both are assigned $T$, (or even if just one of them is assigned $T$), the product is also 3-colorable and thus is assigned $T$. The second condition is simply consequence of the homomorphism being transitive.  
    	
    	The class $M_{\tau}(F)$ then is the class of all 4-critical graphs, those are graphs of chromatic number 4 where each proper subgraph is 3-colorable. On the other hand $M_{\tau}(T)=\{K_1\}$. 
    \end{example}	
    
    If in this example we replace 3-colorability with 2-colorability, then the corresponding $\tau$ is easily verified to be an orthogonality invariant and $M_{\tau}(F)$ is the set of 3-critical graphs, i.e., the set of odd cycles. But if we replace 3-colorability with $k$-colorability for $k\geq 4$, then the corresponding $\tau$ is no longer an orthogonality invariant because Hedetniemi's conjecture is disproved for such values of $k$. In other words, by \cite{T23}, we can find graphs $G$ and $H$ for which $k$-colorability is ``False" but the product admits a $k$-coloring. 
    
    In general the decision function corresponding to the $H$-coloring problem is an orthogonality invariant if and only if $H$ is a multiplicative graphs.

	The set of orthogonality invariant functions is closed under product in the following sense:
	 
	\begin{proposition}
		Let $\tau_1$ be an orthogonality invariant mapping from (di)graphs to $X$ and $\tau_2$ be an orthogonality invariant mapping from (di)graphs to $Y$. Then $\tau=\tau_1 \times \tau_2$ defined by $\tau(P)=(\tau_1(P), \tau_2(P))$ is an orthogonality invariant mapping from (di)graphs to $X\times Y$.
	\end{proposition}

	In the following we will look at the family verifying $(Q \ra \vec P_h \AND Q \not \ra \vec P_{h-1})$.
    	
\subsection{Oriented paths}
\label{subsec:oriented_paths}
	An \emph{oriented} path is an oriented graph whose underlying graph is a path. It is \emph{directed} if all edges are oriented in the same direction.
	The \emph{length} of an oriented path $P$ is its number of arcs. 
	Given an oriented path with a \emph{beginning} $b_{P}$, the \emph{height} of a vertex $v_{i}$ is the difference between the number of forward and backward arcs between $b_{P}$ and $v_{i}$. The \emph{height} of $P$, denoted $h(P)$, is the height of its \emph{end} vertex $e_{P}$. 
	The following is easily observed.
		
	\begin{proposition}\label{morphism preserves height}
		Given oriented paths $P$, and $Q$ and a homomorphism $\phi: P \ra Q$, for any pair $u,v$ of vertices of $P$ we have $h(P_{uv}) = h(\phi(P_{uv}))$ where $P_{uv}$ denotes the subpath going from $u$ to $v$.
	\end{proposition}

	\begin{proof}
    We may assume that $u$ and $v$ are the end vertices of $P$. The claim follows from the following fact. If for vertices $x$ and $y$ of $P$ we have $\phi(x)=\phi(y)$, then $x$ and $y$ are at the same height. Reducing $P$ to a path $P'$ by identifying $x$ and $y$, then continuing this process until there are no two vertices identified by $\phi$ we get a path $P''$ which, first of all, is of the same height as $P$ and, secondly, is isomorphic to $Q$.
	\end{proof}

	\begin{definition}
	\label{def:k-b-path}
		A path of height $k$ is said to be a \emph{$k$-bounded path}, or $k$-b-path for short, if the \emph{beginning}, is the only vertex of height $0$ and the \emph{end} vertex is the only one of height $k$.
	\end{definition}
	It follows immediately that, the arc at the beginning is outgoing, the arc at the end is incoming, and that each of the other vertices are at a height between 1 and $k-1$. When the value of $k$ is of no importance we may refer to the path as a $b$-path. An example of a $4$-b-path is given in Figure~\ref{fig:path}.
	
	\begin{figure}[htb]
 		\centering
 		\scalebox{1}{\begin{tikzpicture}
	\begin{pgfonlayer}{nodelayer}
		\node [style=black empty] (0) at (-4, -3) {};
		\node [style=black empty] (1) at (-3, -1) {};
		\node [style=black empty] (2) at (-2, 1) {};
		\node [style=black empty] (3) at (-1, -1) {};
		\node [style=black empty] (4) at (0, 1) {};
		\node [style=black empty] (5) at (1, 3) {};
		\node [style=black empty] (6) at (2, 1) {};
		\node [style=black empty] (7) at (3, -1) {};
		\node [style=black empty] (8) at (4, 1) {};
		\node [style=black empty] (9) at (5, 3) {};
		\node [style=black empty] (10) at (6, 5) {};
		\node [style=none] (11) at (5, 5) {end};
		\node [style=none] (12) at (-2, -3) {beginning};
		\node [style=none] (13) at (-4, -1) {1};
		\node [style=none] (14) at (-3, 1) {2};
		\node [style=none] (15) at (0, 3) {3};
	\end{pgfonlayer}
	\begin{pgfonlayer}{edgelayer}
		\draw [style=Arc] (0) to (1);
		\draw [style=Arc] (1) to (2);
		\draw [style=Arc] (3) to (2);
		\draw [style=Arc] (3) to (4);
		\draw [style=Arc] (7) to (6);
		\draw [style=Arc] (6) to (5);
		\draw [style=Arc] (4) to (5);
		\draw [style=Arc] (7) to (8);
		\draw [style=Arc] (8) to (9);
		\draw [style=Arc] (9) to (10);
	\end{pgfonlayer}
\end{tikzpicture}}
 		\caption{A 4-b-path}
		\label{fig:path}
	\end{figure}
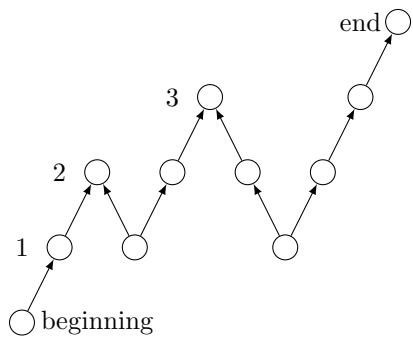
	
	The set of $k$-b-paths is quite relevant to the study of homomorphism because of the following duality: $$\forall G, G \not \ra \vec P_k \iff \exists \text{ a $k$-b-path } P,\text{ such that } P \ra G .$$
	
	\begin{proposition}\label{Prop:morphisms-betweenk-b-paths}
		Given $k$-b-paths $P$ and $Q$ if $P \ra Q$, then $|P| \geq |Q|$. Moreover, if the equality holds, then 
		$P$ and $Q$ are isomorphic.		
	\end{proposition}

	\begin{proof}
		Let $\phi: P \ra Q$ be a homomorphism. By preservation of height, the endpoints of $P$ are mapped to those of $Q$ and since $P$ is connected, $Im(\phi) = Q$ and, in particular, we have $|P| \geq |Im(\phi)| = |Q|$. If $|P| = |Im(\phi)|$, then $\phi$ is injective and, as it is also surjective, it is an isomorphism. 
	\end{proof}
	
	Following this proposition, when there is a homomorphism of $P$ to $Q$ we rather write $Q \la P$, so that the homomorphism order matches the conclusion of $|Q| \leq |P|$. 
	
	The proof of the proposition also implies the following.
	
	\begin{corollary}\label{coro:sujectivityOfk-b-paths}
		Given $k$-b-paths $P$ and $Q$, any homomprphsim of $P$ to $Q$ (if there is one), is surjective and maps beginning to beginning and end to end.
	\end{corollary}

   As a special case of the proposition: There is no homomorphism of a $k$-b-path to a proper subgraph of itself. In other words:
    
	\begin{corollary}\label{coro:K-b-pathsAreCores}
		Every $k$-b-paths is a core.
	\end{corollary}

	In the following we work with disjoint union of $k$-b-paths, however we only consider disjoint unions where the elements are incomparable, i.e., there exists no homomorphism from one to another. More precisely, given a set $P_1, P_2, \dots, P_{\ell}$ of incomparable $k$-b-paths, their \emph{sum}, denoted $\Sigma_P$, is the digraph formed of their disjoint union. For a fixed $k$, the class of $\Sigma_P$ is denoted by ${\cal P}_k$.  We note that the assumption of paths pairwise being incomparable is not essential in our proofs, but it makes statements cleaner.

    Some of the key properties of sums of paths are listed in the followings. The first observation is about checking if there is a homomorphism of $\Sigma_P$ to $\Sigma_Q$.
     
      \begin{observation}\label{obs:HomOfSums}
    	Given two sums of a $k$-b-paths, $\Sigma_P$ and $\Sigma_Q$, there is a homomorphism of $\Sigma_Q$ to $\Sigma_P$ if and only if for each element $Q_j$ of $\Sigma_Q$ there is an element $P_i$ of $\Sigma_P$ such that $P_i \la Q_j$. When there is a homomorphism we write $\Sigma_P \la \Sigma_Q$.
    \end{observation}
      
     That we assumed in a sum of paths the elements forming the sum are pairwise incomparable implies that:
    
    \begin{observation}\label{obs:SumsAreCores}
    Any sum of $k$-b-paths is a core.	
    \end{observation}
    
	\begin{definition} \label{def:sum_paths}
    	Given two sums of $k$-b-paths $\Sigma_{P}$, $\Sigma_{Q}$, the sum $ \Sigma_{P} \vee \Sigma_{Q}$ is the minimum element, $\Sigma_{R}$, of $\cal P_{k}$ such that $\Sigma_{R} \la \Sigma_{P}$ and $\Sigma_{R} \la \Sigma_{Q}$.
	\end{definition}
    Actually, $R$ is a subset of $P \cup Q$ which induces $\core{\Sigma P + \Sigma Q}$
     
     \begin{observation}\label{obs:MappingOfSums}
     	Given a digraph $X$ we have $X \la \Sigma_P \vee \Sigma_Q \iff X \la \Sigma P \AND X \la \Sigma Q $.
     \end{observation} 
     
     \begin{definition}\label{def:mino}
     	Given a sum of $k$-b-paths $\Sigma_P$ we define its \emph{min-order}, denoted $\mino{\Sigma_P}$, to be the smallest order among its building blocks. In other words, $\displaystyle \mino{\Sigma_P}:= \min_{P_i \in P}|P_i|$.
     \end{definition}
 
  This definition allows an extension of the Proposition~\ref{Prop:morphisms-betweenk-b-paths}.
  
  \begin{proposition}
  	Given two sums of $k$-b-paths $\Sigma_P$ and $\Sigma_Q$, if $\Sigma_P \la \Sigma_Q$, then $\mino{\Sigma_P} \leq \mino{\Sigma_Q}$.
  \end{proposition}
 	
  The key reason to work with sums of $k$-b-paths is the following definition of the product of sums of $k$-b-paths in $\cal{P}_k$. 
  
  \begin{definition}\label{def:path_prod}
  	Given sums of $k$-b-paths $\Sigma_P$ and $\Sigma_Q$, their product in $\cal{P}_k$, denoted $\Sigma_P \wedge \Sigma_Q$, is the sum of the $k$-b-paths $R$ where $R$ is taken over all minimal $k$-b-paths (with respect to homomorphism order) having the property that it maps to both $\Sigma_P$ and $\Sigma_Q$.
  \end{definition}

Thus $\Sigma_P \wedge \Sigma_Q$ is distinguished from $\Sigma_P \times \Sigma_Q$ which is the tensor product of $\Sigma_P$ and $\Sigma_Q$ built on  $V(\Sigma_P) \times V(\Sigma_Q)$. That $\Sigma_P \wedge \Sigma_Q$ is well defined and finite follows from the next proposition~\ref{prop:tensorPathProduct}.

	\begin{proposition}\label{prop:tensorPathProduct}
    Given two $k$-b-paths $P$ and $Q$ where $b_P$, $b_Q$ are (respectively) the beginning and $e_P$, $e_Q$ are (respectively) the ends, the product  
	$P \wedge Q $ is the sum of a set of paths in $P \times Q$ with $(b_P, b_Q)$ as the beginning and $(e_P, e_Q)$ as the end. 
    \end{proposition}

    \begin{proof}
	Let $R^*$ be the set of all paths in $P\times Q$ with $(b_P, b_Q)$ as the beginning and $(e_P, e_Q)$ as the end. Observe that the projection of each element of $R^*$ to $P$ (or $Q$) is height preserving, furthermore $(b_P, b_Q)$ is the only vertex mapped to $b_P$ or $b_Q$ in these projections. Similarly $(e_P, e_Q)$ is the only vertex mapped to $e_P$ or $e_Q$. Thus all paths in $R^*$ are $k$-b-paths. Let $R$ be a maximal subset of  $R^*$ whose elements are pairwise incomparable. Hence $\Sigma_R \in \cal P_k$.
	\vspace{3pt}
	
	What remains to show is that if a $k$-b-path $R_i$ maps to both $P$ and $Q$, then it maps to $P\times Q$ in such a way that the beginning is mapped to $(b_P,b_Q)$ and the end is mapped to $(e_P,e_Q)$. Let $\phi$ and $\psi$ be mappings of $R_i$ to $P$ and $Q$, respectively. By \Cref{coro:sujectivityOfk-b-paths}, we know that the beginning of $R_i$ is mapped to $b_P$ by $\phi$ and to $b_Q$ by $\psi$. Similarly, its ends are mapped to $e_P$ by $\phi$ and to $e_Q$ by $\psi$. By Diagram~\ref{fig:DiagramP'Q'} $(\phi, \psi)$ is a mapping of $R_i$ to $P \times Q$ which maps its beginning to $(b_P,b_Q)$ and its end to $(e_P, e_Q)$. Thus the image of $R_i$ by $(\phi, \psi)$ is a path in $P\times Q$ connecting $(b_P,b_Q)$ to $(e_P, e_Q)$.
	\end{proof}
	
	One can now check that for $P = \set{P_1,\dots , P_l}, Q = \set{Q_1, \dots, Q_{l'}}$, we have:
	$$\Sigma_P \wedge \Sigma_Q = \bigvee_{\substack{1\leq i\leq l \\ 1 \leq j \leq l'}} \left ( P_i \wedge Q_j \right ).$$
	It follows that $\wedge$ distributes over $\vee$.
	
An example of the product of two paths is given in \Cref{fig:path_product}. However, one should note that product of two paths could be sum of a number of paths.

	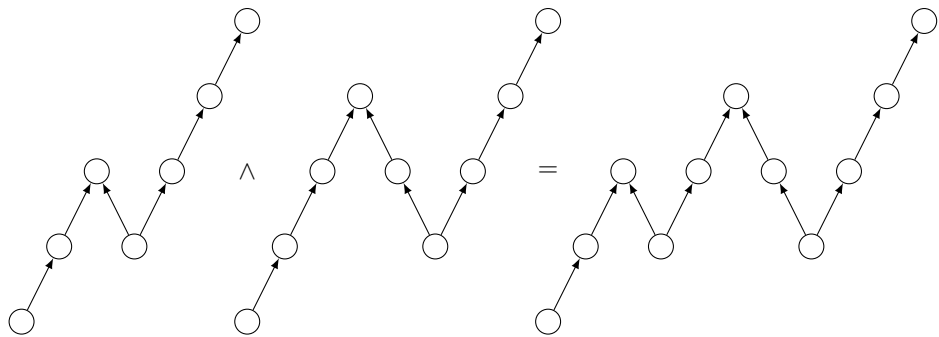
\begin{figure}[htb]
	\centering
	\scalebox{1}{\begin{tikzpicture}
	\begin{pgfonlayer}{nodelayer}
		\node [style=black empty] (0) at (-1, -4) {};
		\node [style=black empty] (1) at (0, -2) {};
		\node [style=black empty] (2) at (1, 0) {};
		\node [style=black empty] (3) at (2, -2) {};
		\node [style=black empty] (4) at (3, 0) {};
		\node [style=black empty] (5) at (4, 2) {};
		\node [style=black empty] (6) at (5, 0) {};
		\node [style=black empty] (7) at (6, -2) {};
		\node [style=black empty] (8) at (7, 0) {};
		\node [style=black empty] (9) at (8, 2) {};
		\node [style=black empty] (10) at (9, 4) {};
		\node [style=black empty] (11) at (-15, -4) {};
		\node [style=black empty] (12) at (-14, -2) {};
		\node [style=black empty] (13) at (-13, 0) {};
		\node [style=black empty] (18) at (-12, -2) {};
		\node [style=black empty] (19) at (-11, 0) {};
		\node [style=black empty] (20) at (-10, 2) {};
		\node [style=black empty] (21) at (-9, 4) {};
		\node [style=black empty] (22) at (-9, -4) {};
		\node [style=black empty] (25) at (-8, -2) {};
		\node [style=black empty] (26) at (-7, 0) {};
		\node [style=black empty] (27) at (-6, 2) {};
		\node [style=black empty] (28) at (-5, 0) {};
		\node [style=black empty] (29) at (-4, -2) {};
		\node [style=black empty] (30) at (-3, 0) {};
		\node [style=black empty] (31) at (-2, 2) {};
		\node [style=black empty] (32) at (-1, 4) {};
		\node [style=none] (33) at (-9, 0) {$\wedge$};
		\node [style=none] (34) at (-1, 0) {=};
	\end{pgfonlayer}
	\begin{pgfonlayer}{edgelayer}
		\draw [style=Arc] (0) to (1);
		\draw [style=Arc] (1) to (2);
		\draw [style=Arc] (3) to (2);
		\draw [style=Arc] (3) to (4);
		\draw [style=Arc] (7) to (6);
		\draw [style=Arc] (6) to (5);
		\draw [style=Arc] (4) to (5);
		\draw [style=Arc] (7) to (8);
		\draw [style=Arc] (8) to (9);
		\draw [style=Arc] (9) to (10);
		\draw [style=Arc] (11) to (12);
		\draw [style=Arc] (12) to (13);
		\draw [style=Arc] (18) to (19);
		\draw [style=Arc] (19) to (20);
		\draw [style=Arc] (20) to (21);
		\draw [style=Arc] (18) to (13);
		\draw [style=Arc] (25) to (26);
		\draw [style=Arc] (29) to (28);
		\draw [style=Arc] (28) to (27);
		\draw [style=Arc] (26) to (27);
		\draw [style=Arc] (29) to (30);
		\draw [style=Arc] (30) to (31);
		\draw [style=Arc] (31) to (32);
		\draw [style=Arc] (22) to (25);
	\end{pgfonlayer}
\end{tikzpicture}}
	\caption{ Exemple of path product in $\cal{P}_k$}
	\label{fig:path_product}
\end{figure}

The definition implies immediately that the product of two sums of $k$-b-paths is a meet of the two in the homomorphism order restricted on the class $\cal{P}_k$. In other words:

\begin{observation}
	Given sums of $k$-b-paths $\Sigma_P$ and $\Sigma_Q$ and $\Sigma_X\in \cal P_k$ we have $\Sigma_P \wedge \Sigma_Q \la \Sigma_X$ if and only if $\Sigma_P \la \Sigma_X$ and $\Sigma_Q \la \Sigma_X$.
\end{observation}

Considering the notion of min-order on sums of $k$-b-paths, we have the following which is implied by \Cref{prop:tensorPathProduct}.

\begin{corollary}
	$\mino{\Sg P \wedge \Sg Q} \leq |\Sg P \times \Sg Q| $
\end{corollary}
   The following follows immediately from the definition.
   
   \begin{observation}
   	If $\Sigma_P \la \Sigma_Q$, then $\Sigma_P\vee \Sigma_Q=\Sigma_P$ and $\Sigma_P \wedge \Sigma_Q=\Sigma_Q$. 
   \end{observation}
   
     For the sake of completing an algebra we allow an empty set to be a $k$-b-path and denote it by $0$ with the convention that $0$ maps to every $k$-b-path.  Moreover, we denote by $1$ the directed path of length $h$ noting that all elements of $\cal{P}_k$, and thus all sums of paths in $\cal{P}_k$ map to $1$. With these notation we have the following whose details of proof we leave to the reader and for related terms we refer to \cite{Birkhoff67}.

      \begin{proposition}
      	The system $({\cal P}_h,\vee,\wedge,0,1)$ is a distributive lattice.
      \end{proposition}

		\subsection{Mountains}

		We will work with a specific set of $k$-b-paths in which every downward move should reach height $1$ before going back up. To give a formal description we first simplify the presentation of a $k$-b-paths $P$ by uniquely associating it to a word in $\set{\up,\dw}^*$. For this, we ignore the first and last arc of $P$, and then consider the sequence of arcs in $P$ as the letters. We note that not every word corresponds to a $k$-b-path, first of all the first and the last letter should be $\up$, secondly the height of vertices should not outreach the limit of $1$ and $k - 1$. To further simplify the notation,  write $\up^i$ for an $i$ number of consecutive $\up$ and similarly $\dw_i$ for an $i$ number of consecutive downward edges. For instance, the word associated with the path in Figure~\ref{fig:path} is $word(P_1) = \up \dw \up^2 \dw_2 \up^2$. Following this terminology, the \emph{concatenation} of two $k$-b-paths $P$ and $Q$, denoted $PQ$, is defined by $word(P.Q) = word(P)word(Q)$. The concatenation of a collection of words is denoted using $\bigodot$.	       
		       
		\begin{definition}
			 Given a sequence $\ell=(l_1,l_2, \cdots, l_r)$ where $1 \leq l_i  \leq k$, an \emph{$\ell$-mountain}, denoted $M_{\ell}$, is the $(k+2)$-b-path whose corresponding word is $\bigodot_{i \leq r} (\up^{l_i}\dw_{l_i}) \up^{k}$. For $i\leq r$ the part of the path corresponding to $\up^{l_i}\dw_{l_i}$ is said to be \emph{$i^{\text th}$-peak}, or simply a $\emph{peak}$. 
			 
			 If $l_1,l_2, \cdots, l_r$ is a decreasing sequence of positive integers, then $M_{\ell}$ is said to be a decreasing mountain.
		 \end{definition}
	 
	 It is easily observed that the word given in the definition satisfies the conditions of being a word for a $(k+2)$-b-path.
	 
	 
	 \begin{definition}\label{def:list_hom} A sequence $\ell=(l_1,l_2, \dots, l_p)$ where $1 \leq l_i  \leq h$ is said to be \emph{homomorphic} to the sequence $r=(r_1, r_2, \dots, r_q)$ if, first of all, $r$ is a subsequence of $\ell$ with $r_i=l_{j_i}$, secondly, $l_{1}, \dots l_{j_{1}} \leq l_{j_1}=r_1$,  thirdly, in the subsequence $l_{j_i}, \dots l_{j_{i+1}}$, $1\leq i\leq q-1$, all values are bounded above by $\max \{l_{j_i}, l_{j_{i+1}}\} = \max \{r_i, r_{i+1}\}$. 
	 \end{definition}
 
	 	In the next proposition we show that being homomorphic between two sequences is the same as existence of a homomorphism between the corresponding mountains.
	 	
	 	\begin{proposition}
			\label{prop:list_hom}
	 		Given two mountains $M_{\ell}$ and $M_{r}$ with their corresponding sequences $\ell=(l_1,l_2, \dots, l_p)$ and $r=(r_1, r_2, \dots, r_q)$ we have $M_{\ell}\to M_{r}$ if and only if $\ell$ is homomorphic to $r$.
	 	\end{proposition} 
 	
 	\begin{proof}
 		\boxed{\Ra} Suppose that $l$ is homomorphic to $r$. We may find a homomorphism of $M_{\ell}$ to $M_{r}$ by mapping all the peaks at $l_{k_i}, \dots l_{k_{i+1}}$ to the $\max \{r_{k_i}, r_{k_{i+1}}\}$, taking an arbitrary choice when $r_{k_i}=r_{k_{i+1}}$. 
		
		\boxed{\La} Suppose $M_\ell$ is homomorphic to $M_r$, we first note that not every homomorphism of $M_{\ell}$ to $M_{r}$ is of the form defined in the previous part. However, to prove this direction of the claim, given a homomorphism, $\psi$, of $M_{\ell}$ to $M_{r}$, using the restriction of $\psi$ on vertices at height 1 of the two mountains we will find a mapping of $\ell$ to $r$. 
 		
 		Let $x_1, x_2, \dots, x_{p+1}$ be the vertices of $M_{\ell}$ at height $1$ and $y_1, y_2, \dots, y_{q+1}$ be the vertices of $M_{r}$ at height 1, each sequence ordered by the distance from the beginning of the path. The mapping $\psi$ must map $x_i$'s to $y_j$'s. 
		Observe that if $\phi(x_i) = y_j$, then there are three possible values for $\phi(x_{i+1})$: 
		\begin{itemize}[noitemsep]
		\item $\phi(x_{i+1}) = y_{j-1} \AND l_i = r_{j-1} \OR$
		\item $ \phi(x_{i+1}) = y_j \AND l_i \leq \max(r_{j-1},r_j) \OR$
		\item $  \phi(x_{i+1}) = y_{j+1} \AND l_j = r_j$ .
		\end{itemize}
		In order for this property to be well defined for $j=1 \AND j=q+1$, we define $r_0 = 0$ and $r_{q+1} = h$ (the height of the mountain).
		We construct a subsequence $(x_{k_j})$ of the $x_i$'s in the following way: 
		\begin{itemize}[noitemsep]
		\item $k_1$ is the maximal index such that $\psi(x_{k_1}) = y_1$ and all peaks until $l_{k_1}$ have height at most $r_1$.
		\item $k_{j+1}$ is the maximal index bigger than $k_j$ such that $\psi(x_{k_{j+1}}) = y_{j+1}$ and all peaks between $l_{k_j} \AND l_{k_{j+1}}$ have height at most $\max(r_j,r_{j+1})$.

		\end{itemize}

		What remains to show is that the sequence of $k_j$'s is well defined. To this end, all one needs to do is to show that, having defined $k_j$, there exists a $t> k_j$ such that $\psi(x_t)=y_{_{j+1}}$. For this, in fact we show that $\psi(x_{k_{_{j+1}}})=y_{_{j+1}}$. Towards a contradiction, suppose that $\psi(x_{k_j\!+1})\neq y_{j+1}$, then $\psi(x_{k_{j}\!+1})\in \set{x_{k_{j-1}}, x_{k_{j}} }$ and, furthermore, there exists a first index $t > k_j$, such that $\psi(x_t)=y_{j+1}$. Then $\set{ x_{k_{j}} , \dots, x_{t-1}}$ is surjectively mapped to $\set{y_{s}, \dots, y_{j}}$ for some $s\leq j$. Let $r_{s'}$ be the largest height of a peak between $y_{s}$ and $y_j$. Hence, for all $i \in [k_{s'}, t[$, $l_i \leq r_{s'}$. But in our process of choosing $k_{s'}$, we should have chosen a larger index. This contradiction verifies the claim that $\psi(x_{k_j\!+1}) = y_{j+1}$, completing the proof.
 	\end{proof}

An example of a homomorphism between two mountains trough homomorphic sequences is given in Figure~\ref{fig:morphism_mountain}.

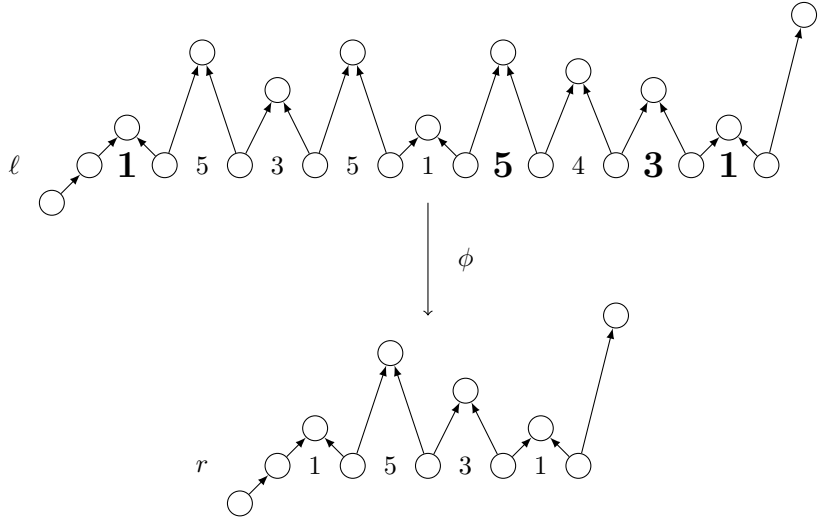
\begin{figure}[htb]
	\centering
	 \scalebox{1}{\begin{tikzpicture}
	\begin{pgfonlayer}{nodelayer}
		\node [style=black empty] (0) at (-10, 1) {};
		\node [style=black empty] (1) at (-9, 2) {};
		\node [style=black empty] (2) at (-8, 3) {};
		\node [style=black empty] (3) at (-7, 2) {};
		\node [style=black empty] (4) at (-6, 5) {};
		\node [style=black empty] (5) at (-5, 2) {};
		\node [style=black empty] (6) at (-4, 4) {};
		\node [style=black empty] (7) at (-3, 2) {};
		\node [style=black empty] (8) at (-5, -7) {};
		\node [style=black empty] (9) at (-4, -6) {};
		\node [style=black empty] (10) at (-3, -5) {};
		\node [style=black empty] (11) at (-2, -6) {};
		\node [style=black empty] (12) at (-1, -3) {};
		\node [style=black empty] (13) at (0, -6) {};
		\node [style=black empty] (14) at (1, -4) {};
		\node [style=black empty] (15) at (2, -6) {};
		\node [style=black empty] (16) at (4, -6) {};
		\node [style=black empty] (17) at (5, -2) {};
		\node [style=black empty] (18) at (3, -5) {};
		\node [style=black empty] (19) at (-2, 5) {};
		\node [style=black empty] (20) at (-1, 2) {};
		\node [style=black empty] (21) at (1, 2) {};
		\node [style=black empty] (22) at (3, 2) {};
		\node [style=black empty] (23) at (5, 2) {};
		\node [style=black empty] (24) at (7, 2) {};
		\node [style=black empty] (25) at (0, 3) {};
		\node [style=black empty] (26) at (2, 5) {};
		\node [style=black empty] (27) at (4, 4.5) {};
		\node [style=black empty] (28) at (6, 4) {};
		\node [style=black empty] (29) at (8, 3) {};
		\node [style=black empty] (30) at (9, 2) {};
		\node [style=black empty] (31) at (10, 6) {};
		\node [style=none] (32) at (0, 1) {};
		\node [style=none] (33) at (0, -2) {};
		\node [style=none] (34) at (-8, 2) {};
		\node [style=none] (35) at (-6, 2) {};
		\node [style=none] (36) at (-4, 2) {};
		\node [style=none] (37) at (-8, 2) {\Large \bf 1};
		\node [style=none] (38) at (-6, 2) {5};
		\node [style=none] (39) at (-4, 2) {3};
		\node [style=none] (40) at (-2, 2) {5};
		\node [style=none] (41) at (0, 2) {1};
		\node [style=none] (42) at (2, 2) {\Large \bf 5};
		\node [style=none] (43) at (4, 2) {4};
		\node [style=none] (44) at (6, 2) {\Large\bf 3};
		\node [style=none] (45) at (8, 2) {\Large\bf 1};
		\node [style=none] (46) at (-3, -6) {1};
		\node [style=none] (47) at (-1, -6) {5};
		\node [style=none] (48) at (1, -6) {3};
		\node [style=none] (49) at (3, -6) {1};
		\node [style=none] (50) at (1, -0.5) {$\phi$};
		\node [style=none] (51) at (-11, 2) {$\ell$};
		\node [style=none] (52) at (-6, -6) {$r$};
	\end{pgfonlayer}
	\begin{pgfonlayer}{edgelayer}
		\draw [style=Arc] (0) to (1);
		\draw [style=Arc] (1) to (2);
		\draw [style=Arc] (3) to (2);
		\draw [style=Arc] (3) to (4);
		\draw [style=Arc] (5) to (4);
		\draw [style=Arc] (5) to (6);
		\draw [style=Arc] (7) to (6);
		\draw [style=Arc] (8) to (9);
		\draw [style=Arc] (9) to (10);
		\draw [style=Arc] (11) to (10);
		\draw [style=Arc] (11) to (12);
		\draw [style=Arc] (13) to (12);
		\draw [style=Arc] (13) to (14);
		\draw [style=Arc] (15) to (14);
		\draw [style=Arc] (15) to (18);
		\draw [style=Arc] (16) to (18);
		\draw [style=Arc] (16) to (17);
		\draw [style=Arc] (7) to (19);
		\draw [style=Arc] (20) to (19);
		\draw [style=Arc] (20) to (25);
		\draw [style=Arc] (21) to (26);
		\draw [style=Arc] (22) to (27);
		\draw [style=Arc] (23) to (28);
		\draw [style=Arc] (24) to (29);
		\draw [style=Arc] (30) to (31);
		\draw [style=Arc] (30) to (29);
		\draw [style=Arc] (24) to (28);
		\draw [style=Arc] (23) to (27);
		\draw [style=Arc] (22) to (26);
		\draw [style=Arc] (21) to (25);
		\draw [style=automorphism] (32.center) to (33.center);
	\end{pgfonlayer}
\end{tikzpicture}}
   	\caption{ Bold numbers in $\ell$ represent $r$ as a homomorphic subsequence.}
   	\label{fig:morphism_mountain}

\end{figure}

	The family of mountains will be useful as it is a large family of paths remaining easy to handle .

\section{Building a product that is a core}\label{sec:BuildingProduct}

To address the problem of building graphs whose product has a large core, in this section we will in fact provide a sufficient condition under which product of a set of cores itself is a core. To motivate our approach let us consider this observation: suppose $G_1$ is a digraph with a source vertex $s$ and $D_2$ is a digraph with a sink vertex $t$. Then in $G_1\times G_2$ the vertex $(s,t)$ is an isolated vertex ensuring that this digraph is not a core. To avoid such cases we will work with digraphs with an universal vertex.

\begin{definition}
	\label{def:cone}
	Given a digraph $G$, the digraph $\cone G$ is built from $G$ by adding  a universal vertex $\apex{G}$ that is linked in both directions with every vertex of $G$. 
\end{definition}

For example, $\cone{\vec P_2}$ is given in \Cref{fig:coneP2}. The usefulness of this construction comes from the fact that $|\cone G| = |G| + 1$ but $|\core{\cone G\times \cone H}|$ can be much larger than $|\core{G\times H}|$. This is the key fact in Lemma~\ref{lem:conecommute} and Theorem \ref{thm:VSCmultiple}.    

\begin{figure}[h]
\captionsetup{justification=raggedright} 

	\begin{minipage}[b]{0.4\textwidth}
		\centering
		\scalebox{1.2}{\begin{tikzpicture}
	\begin{pgfonlayer}{nodelayer}
		\node [style=small black empty] (0) at (-2, -2.5) {};
		\node [style=small black empty] (1) at (2, -2.5) {};
		\node [style=small black empty] (2) at (0, 1.25) {};
		\end{pgfonlayer}
	\begin{pgfonlayer}{edgelayer}
		\draw [style=Arc] (2) to (0);
		\draw [style=Arc] (0) to (2);
		\draw [style=Arc] (2) to (1);
		\draw [style=Arc] (1) to (2);
		\draw [style=Arc] (0) to (1);
	\end{pgfonlayer}
\end{tikzpicture}}
	\end{minipage}
	\begin{minipage}{.6\textwidth}
		\centering
		\scalebox{1}{\begin{tikzpicture}
	\begin{pgfonlayer}{nodelayer}
		\node [style=none] (0) at (-6, 7.75) {$(\star_{_{G}}, \star_{_H})$};
		\node [style=none] (1) at (0, 7.25) {};
		\node [style=none] (2) at (-5, 7.25) {};
		\node [style=none] (3) at (-5.75, 6) {};
		\node [style=none] (4) at (-5.75, 1) {};
		\node [style=none] (5) at (-5, 6) {};
		\node [style=none] (6) at (0, 6) {};
		\node [style=none] (7) at (0, 1) {};
		\node [style=none] (8) at (-5, 1) {};
		\node [style=none] (9) at (-3, 2.5) {$G \times H$};
		\node [style=small black empty] (10) at (-2.5, 7.75) {};
		\node [style=small black empty] (11) at (-1, 7.75) {};
		\node [style=small black empty] (12) at (-2.5, 7) {};
		\node [style=small black empty] (13) at (-1, 4.5) {};
		\node [style=small black empty] (14) at (-1, 3.5) {};
		\node [style=small black empty] (15) at (-1, 2.5) {};
		\node [style=small black empty] (16) at (-1, 1.5) {};
		\node [style=small black empty] (17) at (-6, 7) {};
		\node [style=none] (18) at (-4, 7) {};
		\node [style=none] (19) at (-6, 5) {};
		\node [style=none] (20) at (-6, 4.5) {};
		\node [style=none] (21) at (-6, 4) {};
		\node [style=none] (22) at (-3.5, 7) {};
		\node [style=none] (23) at (-3, 7) {};
		\node [style=none] (24) at (-3.25, 5) {};
		\node [style=none] (25) at (-4, 4.75) {};
		\node [style=none] (26) at (-4, 4.25) {};
		\node [style=none] (27) at (-3.5, 4.5) {};
		\node [style=none] (28) at (-3.5, 5.5) {};
		\node [style=none] (29) at (-4.5, 4.5) {};
		\node [style=none] (30) at (1.5, 7) {$(\star_{_{G}}\times H)$};
		\node [style=none] (31) at (-6, 0.5) {$(G \times \star_{_H})$};
		\node [style=none] (32) at (-2.5, 8.25) {$H$};
		\node [style=none] (34) at (0, 6.75) {};
		\node [style=none] (35) at (-5, 6.75) {};
		\node [style=none] (36) at (-6.25, 6) {};
		\node [style=none] (37) at (-6.25, 1) {};
		\node [style=small black empty] (38) at (-1, 5.5) {};
		\node [style=none] (39) at (-3.75, 5) {};
	\end{pgfonlayer}
	\begin{pgfonlayer}{edgelayer}
		\draw (2.center) to (1.center);
		\draw (3.center) to (4.center);
		\draw (5.center) to (6.center);
		\draw (6.center) to (7.center);
		\draw (7.center) to (8.center);
		\draw (8.center) to (5.center);
		\draw [style=Arc] (10) to (11);
		\draw [style=Arc] (12) to (13);
		\draw [style=Arc] (12) to (14);
		\draw [style=Arc] (12) to (15);
		\draw [style=Arc] (12) to (16);
		\draw [style=Arc] (22.center) to (19.center);
		\draw [style=Arc] (19.center) to (23.center);
		\draw [style=Arc] (18.center) to (19.center);
		\draw [style=Arc] (20.center) to (18.center);
		\draw [style=Arc] (21.center) to (18.center);
		\draw [style=Arc] (21.center) to (22.center);
		\draw [style=Arc] (21.center) to (23.center);
		\draw [style=Arc] (20.center) to (22.center);
		\draw [style=Arc] (20.center) to (23.center);
		\draw [style=Arc] (22.center) to (20.center);
		\draw [style=Arc] (23.center) to (21.center);
		\draw [style=Arc] (17) to (24.center);
		\draw [style=Arc] (17) to (25.center);
		\draw [style=Arc] (17) to (26.center);
		\draw [style=Arc] (26.center) to (17);
		\draw [style=Arc] (25.center) to (17);
		\draw [style=Arc] (24.center) to (17);
		\draw [style=Arc] (17) to (29.center);
		\draw [style=Arc] (17) to (28.center);
		\draw [style=Arc] (17) to (27.center);
		\draw [style=Arc] (29.center) to (17);
		\draw [style=Arc] (28.center) to (17);
		\draw (35.center) to (34.center);
		\draw [style=Edge] (1.center) to (34.center);
		\draw [style=Edge] (2.center) to (35.center);
		\draw (36.center) to (37.center);
		\draw [style=Edge] (37.center) to (4.center);
		\draw [style=Edge] (36.center) to (3.center);
		\draw [style=Arc] (12) to (38);
		\draw [style=Arc] (17) to (39.center);
	\end{pgfonlayer}
\end{tikzpicture}}
	\end{minipage}
	\par
	\begin{minipage}[t]{.4\textwidth}  
		\caption{$\cone{\vec P_2}$}
		\label{fig:coneP2}
	\end{minipage}
	\begin{minipage}[t]{.6\textwidth}  
		\caption{Diagram of $\cone G \times \cone H$}
		Observe that the digons induces a complete bipartite graph on both $G \times H \cup \set{(\apex G, \apex H)}$  and $(G \times \apex H)  \cup (\apex G \times H)$
		\label{fig:product_cones}
	\end{minipage}
\end{figure}

In working with $\cone{G}\times \cone{H}$, we will partition the vertices of $\cone G \times \cone H$ into four parts: $$\set{(\apex G, \apex H)},\ G \times \apex H  = \set{(u, \apex H) \mid u \in G},\ \apex G \times H = \set{(\apex G, x) \mid x \in H},\text{ and }V(G \times H).$$ 
If $G$ and $H$ are oriented (no digon), the connection between the sets is as in \Cref{fig:product_cones}.
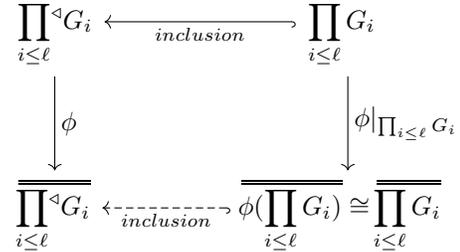
\begin{wrapfigure}[14]{r}{0.4\textwidth}
	\vspace{.3cm}
	\centering
	\begin{tikzcd}	
		\displaystyle \prod_{i \leq \ell} \cone G_i \arrow[dd, "\displaystyle \phi"]     &       & \displaystyle \prod_{i \leq \ell} G_i \arrow[ll, "inclusion",hook] \arrow[dd, " \displaystyle \phi|_{\prod_{i \leq \ell} G_i }", shift left] \\
															&  &                                                                              \\
		\core{\prod_{i \leq \ell} \cone G_i}  &       &   \core{\phi(\prod_{i \leq \ell} G_i )} \cong \core{\prod_{i \leq \ell} G_i } \arrow[ll, hook,"inclusion", dashed]               
	\end{tikzcd}
	\caption*{Commutative diagram of \Cref{lem:conecommute}}
\end{wrapfigure}
\begin{lemma} \label{lem:conecommute}
	Let $G_1, \dots, G_\ell$ be  a finite sequence of nontrivial oriented graphs. The retract of $\prod_{i \leq \ell} \cone G_i$ to a core of itself induces a homomorphism of $\prod_{i \leq \ell} G_i$ to itself.
	In particular, the diagram on the right commutes.
\end{lemma}

\begin{proof}
	In the \Cref{lem:conecommute}, the top horizontal arrow is the natural inclusions induced by the inclusion of $G_i$ in $\cone G_i$.
	Let $\phi$ be a retract of $\prod_{i \leq \ell} \cone G_i$ to a core of itself. 
	Because each $G_i$ is nontrivial, by considering induced subgraphs isomorphic to $\cone \vec P_2$, we have $\phi(\apex {G_1},\dots,\apex {G_\ell}) = (\apex {G_1},\dots,\apex {G_\ell})$.
	Then the neighborhood of $(\apex {G_1},\dots,\apex {G_\ell})$ is $(V(\prod_{i \leq \ell} G_i))$ and must map to the neighborhood of $\phi(\apex {G_1},\dots,\apex {G_\ell})$. In other words $\phi(\prod_{i \leq \ell} G_i) \subseteq (\prod_{i \leq \ell} G_i)$.
	And since $\phi$ is a retract $\core{\phi(\prod\limits_{i \leq \ell} G_i)} \cong \core{\prod\limits_{i \leq \ell} G_i}$. 
\end{proof}

The following is the key tool towards building a family of digraphs whose product is a core.

\begin{theorem}
	For any pair $G$ and $H$ of orthogonal and incomparable oriented graphs the product of their cones, $\cone G \times \cone H$, is a core.
\end{theorem}

\begin{proof}
	Let $\phi$ be a retract of $\cone G \times \cone H$ to (one of) its core. The claim of the theorem is that $\phi$ is the identity. We will prove this separately on the four parts $\set{(\apex G, \apex H)}$, $ G \times \apex H$, $\apex G \times  H$, and $V(G \times H)$.
	
	For the first part, we have already seen in \Cref{lem:conecommute} that $\phi(\apex G,\apex H) = (\apex G,\apex H)$.
	Consider the restriction of $\phi$ on the second part, i.e.,  $G \times \apex H $. We first show that $\phi$ maps this set of vertices to itself.
	Toward a contradiction, suppose a vertex $(u, \apex H)$ of this set is mapped to another part. Thus we have three possibilities.

	\begin{itemize}
		\item $\phi(x,\apex H)  = (\apex G,y)$: In this case, since $\phi$ is a retract and $(\apex G,y) \in Im(\phi)$, we have that $\phi(\apex G,y) = (\apex G,y)$ which is impossible because $(x,\apex H) \digon (\apex G,y)$ and $\phi^2 = \phi$.
	
		\item $\phi(u,\apex H)=(\apex G,\apex H) $: In this case, since $ (u, \apex H)$ is adjacent to every vertex in $\apex G \times H$ (by digons) and since the neighborhood of $(\apex G, \apex h)$ is $G \times H$, we have $\phi(\apex G \times H) \subset G \times H$. Furthermore, any other $(v, \apex H)$ is also adjacent with digons to $\apex G \times H$, therefore $\phi(v, \apex H) = (\apex G, \apex H)$ \emph{i.e.} $G \times \apex H = (\apex G, \apex H)$.
		
		Next we claim that if $\phi(\apex G, x) = (v,x')$, then $x'$ dominates $x$ (i.e. $\vec N (x) \subset \vec N (x')$ and  $\rvec{N} (x) \subset \rvec{N} (x')$). As $H$ is a core, this would imply $x=x'$.
		Let $y$ be a neighbor of $x$, and, by symmetry,  assume that $y$ is an out-neighbour of $x$ .
		
		By \Cref{lem:conecommute}, we have that $\phi|_{G \times H}$ is a retract of $G \times H$. By \Cref{prop:project+hom} there exists a vertex $w$ such that $(w,y) \in Im(\phi)$.
		Recall that $\phi(\apex G, x) = (v,x')$, moreover, since $\phi$ is a retract and $(w,y)\in Im(\phi)$, $\phi(w,y) = (w,y)$. Since $(w,y)$ is an out-neighbour of $(\apex G, x)$ in $\cone G \times \cone H$ we have $(w,y) \tikzRarrow (v,x')$ which means, in particular, that $y$ is an out-neighbour of $x'$. This complete the proof of $x'$ dominating $x$ implying $\phi(\apex G,x) = (v,x)$.
		
		Now we claim that $H \ra G$ which is in contradiction with $H$ and $G$ being incomparable.
		Let $x \tikzRarrow y$  be an arc of $H$. Let $\phi(\apex G,x) = (v,x)$, $\phi(\apex G,y) = (w,y)$. Observe that $(\apex G, x) \tikzRarrow (w,y)$. Applying $\phi$ and noting that it is a retract, we have $(v,x) \tikzRarrow (w,y)$ and thus $v \tikzRarrow w$. This proves that the restriction of $\phi$ on $\apex G \times H$ composed with the projection of $G \times H$ to $G$ is a homomorphism of $H$ to $G$.
		
		\item $\phi(x,\apex H) \in G \times H$: In this case, considering the digons in \Cref{fig:product_cones} connecting $\apex G \times H$ and $G \times \apex H$, we conclude that $\phi(\apex G, \times H) = (\apex G,\apex H)$. Symmetrically to the above case, we reach a contradiction.
	\end{itemize}		
		Having proved $\phi(G\times \apex H) \subset G\times \apex H$ we aim to prove that $\phi|_{G\times \apex H}$ is the identity. Let $\phi(u,\apex H) = (u',\apex H)$. Once again, we show that $u'$ dominates $u$ which would imply $u'=u$.
		To that end, consider a neighbor $v$ of $u$, and, by symmetry, assume $u \tikzRarrow v$. Then, by \Cref{lem:conecommute} and \Cref{prop:project+hom}, there exists $(v,y) \in Im(\phi)$. Observe that $(u, \apex H) \tikzRarrow (v,y)$, applying $\phi$ which is a retract, we get $(u', \apex H) \tikzRarrow (v,y)$. In particular $u' \tikzRarrow v$ meaning that $u'$ dominates $u$.
		
		By symmetry $\phi|_{\apex G\times H}$ is also the identity.
		To see that $\phi|_{G\times H}$ is the identity, we note that the neighborhood of a vertex $(u,x)$ in $\apex G\times H \cup G\times \apex H$ uniquely determines $(u,x)$. 
\end{proof}

So far, we have provided a sufficient condition for the product of two digraphs to be a core. In the next theorem, we do the same for an arbitrary number of digraphs, but for the theorem to work, the condition we require are stronger.

\begin{theorem}\label{thm:VSCmultiple}
	If $\cal G = \set{G_1,\dots ,G_\ell}$ is a family of oriented incomparable cores verifying:
	\begin{enumerate}
		\item[(a)] $\forall G \in \cal G$, we have $ \prod\limits_{\substack{H \in \cal G \\ H \neq G}}H \not \ra G$ and $\prod\limits_{\substack{H \in \cal G \\ H \neq G}}H \ra \cone G$
		\item[(b)] $\cal G$ is an orthogonal family: $ \forall G \in \cal G$, we have $ \prod\limits_{H \in \cal G}H \rra G$, 
	\end{enumerate}
	then $ \prod\limits_{G \in \cal G} \cone G \tx{ is a core }$
\end{theorem}

\begin{proof}
	Consider a retract $\phi$ of $ \prod\limits_{G \in \cal G} \cone G$ to itself, we shall prove that $\phi$ is the identity.
	
	We consider vertices of the product where each coordinate except one is an apex and refer to them as \emph{basis elements}.
	By symmetry of the coordinates, we work with vertices of the form $(x_1,\apex {G_2} \dots \apex{G_{\ell}})$ and show that:

    \vspace{5pt}
    \textbf{Claim 1:} The restriction of $\phi$ onto basis elements is the identity.\\
	We first show that  $\phi(x_1,\apex{G_2}, \dots, \apex{G_{\ell}}) \not \in \apex{G_1}\! \times \cone G_2 \times \dots \times \cone G_\ell  $. Assume to the contrary that $\phi(x_1,\apex{G_2}, \dots, \apex{G_{\ell}}) = (\apex{G_1},\alpha_2,\dots,\alpha_\ell)$, observe that $\apex{G_1}\! \times G_2 \times \dots \times G_\ell$ is the set of vertices each adjacent to the vertex $(x_1,\apex{G_2} ,\dots, \apex{G_{\ell}})$ with a digon. Therefore, for each element $(\apex{G_1}, x_2, \dots , x_\ell)$ of $\apex{G_1}\! \times G_2 \times \dots \times G_\ell$, its image under $\phi$ is adjacent to $\phi(x_1,\apex{G_2}, \dots, \apex{G_{\ell}})$ with a digon. It follows that $\phi(\apex{G_1}, x_2, \dots , x_\ell) \in G_1 \times \cone G_2 \times \dots \times \cone G_\ell$. In other words $\pi_1 \circ \phi(\apex{G_1}\! \times G_2 \times \dots \times G_\ell) \subseteq V(G_1)$.
		Independently, using \Cref{lem:conecommute}, we have $\pi_1 \circ \phi(G_1\! \times G_2 \times \dots \times G_\ell) \subseteq V(G_1)$.
		
		Combining the two together, we conclude that $\pi_1 \circ \phi$ maps $\cone G_1\! \times G_2 \times \dots \times G_\ell$ to $ G_1$. 
		But by the second part of the assumptions $(a)$, we have $G_2 \times \dots \times G_\ell \ra \cone G_1$. Thus $\cone G_1\! \times G_2 \times \dots \times G_\ell$ is homomorphically equivalent to $G_2 \times \dots \times G_\ell$. Hence we have that $G_2 \times \dots \times G_\ell \to G_1$, however this contradicts the fist part of assumption $(a)$.
		
		\vspace{5pt}
		
		To complete the proof of Claim 1, assume $\phi(x_1,\apex{G_2} \dots \apex{G_{\ell}}) \in x_1'\! \times \cone G_2 \times \dots \times \cone G_\ell  $, with $x_1' \in G_1$. 
	    We aim to prove that $x_1'$ dominates $x_1$ which would imply $x_1'=x_1$. Toward that, consider a neighbor $y_1$ of $x_1$ in $G_1$. By symmetry, assume $x_1 \tikzRarrow y_1$. By assmuption \emph{(b)} and using the \Cref{prop:project+hom}, we know the existence of $y_2, \dots,y_\ell$ such that $(y_1,y_2,\dots, y_\ell) \in Im(\phi)$.
	    Observe that $(x_1,\apex{G_2} \dots \apex{G_{\ell}}) \tikzRarrow (y_1,y_2,\dots, y_\ell)$. Applying $\phi$, and noting that it is a retract, we get $\phi(x_1,\apex{G_2} \dots \apex{G_{\ell}})\tikzRarrow (y_1,y_2,\dots, y_\ell)$. Then looking at the first coordinate, we get $x_1' \tikzRarrow y_1$, showing that $x_1'$ dominates $x_1$.
		
		\vspace{5pt}
		
	\textbf{Claim 2:} $\phi|_{G_1 \times \dots \times G_\ell}$ is the identity.\\
	Let $(x_1,\dots,x_\ell) \in G_1 \times \dots \times G_\ell$. By the \Cref{lem:conecommute}, we already know that $\phi(x_1,\dots,x_\ell) \in G_1 \times \dots \times G_\ell$. Hence, we assume $\phi(x_1,\dots,x_\ell) = (x_1',\dots,x_\ell') $.
	By the symmetry of the coordinates, it is enough to show that $x_1'$ dominates $x_1$, and since $G_1$ is a core, it would imply $x_1'=x_1$.
	
	Let $y$ be a neighbor of $x_1$ in $G_1$, and by symmetry, assume that $x_1 \tikzRarrow y$. Then $(x_1,\dots,x_\ell) \tikzRarrow (y, \apex{G_2}, \dots, \apex{G_{\ell}})$.
	By applying $\phi$ we have $(x_1',\dots,x_\ell') \tikzRarrow \phi(y, \apex{G_2}, \dots, \apex{G_{\ell}})$. In particular, by restriction to the first coordinate, we conclude $x'_1 \tikzRarrow y$ meaning that $x_1'$ dominates $x_1$ and thus $x_1' = x_1$. Similarly, we have the same for all other coordinates, therefore $\phi(x_1,\dots,x_\ell) = (x_1,\dots,x_\ell)$, and $\phi$ is the identity on $G_1 \times \dots \times G_\ell$.
	\vspace{5pt}
	
	Next, we have to prove the same on the remaining vertices of $\prod\limits_{G \in \cal G} \cone G$. 
	For this, we will need to pay close attention to the vertices with empty out-neighborhood or in-neighborhood in $G_i$. We write $\vec{N}_{G_i}(\alpha_i)$ (\emph{resp.} $\rvec{N}_{G_i}(\alpha_i)$) for the out-neighborhood in $G_i$  (\emph{resp.} in-neighborhood) of $\alpha_i \in \cone G_i$. In particular $\vec{N}_{G_i}(\apex{G_i}) = V(G_i)$.
	
	Since the graphs are cores, they cannot have any isolated vertices, except if $\cal G = \set{K_1}$. 
	Furthermore, we can assume that each $G_i$ contains a $\vec P_3$ as a subgraph. Otherwise, since they are incomparable cores, this would imply $\cal G = \set{\vec P_2}$. Therefore, in every $G_i$, there is a vertex $z_i$ such that $N_{G_i}^\ra(z_i) \neq \emptyset$ and $N_{G_i}^\la(z_i) \neq \emptyset$.
	
	\textbf{Claim 3:} If $\phi(\alpha_1,\dots,\alpha_\ell) = (\alpha_1',\dots,\alpha_\ell')$ then $\alpha_i' = \apex{G_i} \iff \alpha_i = \apex{G_i}$.\\
	We consider several cases:
	\begin{enumerate}
	\item[I.] Assume $\alpha_1 = \apex{G_1}$ and $\vec{N}_{G_i}(\alpha_i) \neq \emptyset$ for $i=2, \ldots, \ell$. Toward a contradiction, assume $\alpha_1' = x_1'\in G_1$.
	For $i \neq 1$, chose $y_i \in G_i$ such that $\alpha_i \tikzRarrow y_i$. Then $(\apex{G_1},\alpha_2,\dots,\alpha_\ell) \tikzRarrow (x_1',y_2,\dots,y_\ell)$. From Claim 2., we have that $(x_1',y_2,\dots,y_\ell)$ is fixed by $\phi$, therefore $\phi(\apex{G_1},\alpha_2\dots,\alpha_\ell) = (x_1',\alpha_2'\dots,\alpha_\ell') \tikzRarrow (x_1',y_2,\dots,y_\ell)$. That would imply a loop on $x_1'$ in $G_1$ which is a contradiction.
	
	\item[II.] Assume $\alpha_1 = x_1 \neq \apex{G_1}$.
	Let $\beta_i$ be a vertex of $\cone G_1$ such that $\alpha_i \digon \beta_i$ and $N_{G_i}^\ra(\beta_i) \neq \emptyset$. It exists because one can chose $\apex{G_i}$ if $\alpha_i \neq \apex{G_i}$, or $z_i$ if $\alpha_i = \apex{G_i}$. Observe that 
	$(x_1,\alpha_2,\dots,\alpha_\ell) \digon (\apex{G_1},\beta_2,\dots,\beta_\ell)$. We have already seen in I that $\phi(\apex{G_1},\beta_2\dots,\beta_\ell)$ preserves the first coordinate, and since $\phi(\apex{G_1},\beta_2\dots,\beta_\ell) \digon \phi(x_1,\alpha_2,\dots,\alpha_\ell)$, we have $\alpha_1' \digon \apex{G_1}$ and, in particular, $\alpha_1' \neq \apex{G_1}$
	
	\item[III.] Assume $\alpha_1 = \apex{G_1}$,  with no assumption on the $\alpha_i$ this time. Using the same $\beta_i$ as in II and any $x_1 \in G_1$, since $(\apex{G_i},\alpha_2,\dots,\alpha_\ell) \digon (x_1,\beta_2,\dots,\beta_\ell)$, we have $\phi(\apex{G_i},\alpha_2,\dots,\alpha_\ell) \digon \phi(x_1,\beta_2,\dots,\beta_\ell)$. We have  $\alpha_1 \digon \pi_{1} \circ \phi(x_1,\beta_2,\dots,\beta_\ell)$. By II, $\pi_{1} \circ \phi(x_1,\beta_2,\dots,\beta_\ell) \ni G_1$ and, therefore, $\alpha_1 = \apex{G_1}$.
	
	\end{enumerate}
	
	 
	\textbf{Claim 4:}  Assume $\phi(x_1,\alpha_2,\dots,\alpha_\ell) = (\alpha_1',\alpha_2',\dots,\alpha_\ell')$. If $\forall i, i\neq 1,  \vec N_{G_i}(\alpha_i) \neq \emptyset$ and $\rvec{N}_{G_i}(\alpha_i) \neq \emptyset$, then $\alpha_1' = x_1$.\\
	 By Claim 3, we know that $\alpha_1' \in G_1$ hence we write $\alpha_1' = x_1'$.
	 Consider a neighbor $y_1$ of $x_1$. By symmetry, assume $x_1\tikzRarrow y_1$. Then for $i>1$, chose $y_i \in G_i$ such that $\alpha_i \ra y_i$.
	 Then $(x_1,\alpha_2\dots,\alpha_\ell) \tikzRarrow (y_1,\dots,y_\ell)$. By applying $\phi$, noting that, by Claim 2, $\phi$ fixes $(y_1,\dots,y_\ell)$, we get $(x_1',\alpha_2',\dots,\alpha_\ell') \tikzRarrow  (y_1,\dots,y_\ell)$. In particular, $x_1' \tikzRarrow y_1$. 
	 Therefore $x_1'$ dominates $x_1$ in $G_1$ and hence $x_1'= x_1$.
	 \vspace{5pt}
	
	To prove the theorem, we consider $\phi(\alpha_1, \dots, \alpha_\ell) = (\alpha_1', \dots, \alpha_\ell')$. If $\alpha_1 = \apex{G_1}$ then, by Claim 3, $\alpha_1' = \apex{G_1}$. Otherwise let $\alpha_1 = x_1$ then by Claim 3, $\alpha_1' = x_1'$.
	Note that now, we have no restriction on the $\alpha_i$'s. Consider $y_1$ a neighbor of $x_1$, by symmetry, assume $x_1 \tikzRarrow y_1$. For $i\neq 1$, let $\beta_i \in \set{\apex{G_i}, z_i}$ such that $\alpha_i \digon \beta_i$. Then $\phi(x_1,\alpha_2\dots,\alpha_\ell) \tikzRarrow \phi(y_1,\beta_2,\dots,\beta_\ell)$. Observing that $\beta_i$'s satisfy the condition of Claim 4, $y_1$ is the first coordinate of $\phi(y_1,\beta_2,\dots,\beta_\ell)$. Hence $x_1' \tikzRarrow y_1$. Therefore $x_1'$ dominates $x_1$ in $G_1$ thus $x_1'=x_1$. 
\end{proof}

\section{Decreasing mountains}
\label{sec:Decreasing_mountains}

In the previous section, we have provided conditions under which the tensor product of a family of digraphs is a core. Here we provide examples. 
\begin{definition}
	A mountain $M_{\partial}$ is said to be \emph{decreasing} if the corresponding sequence $\partial = (\partial_1, \dots, \partial_\ell)$ is (strictly) decreasing. The set of all decreasing mountains of height $h$ with $\ell$ peaks satisfying $\partial_1 = h-2$ is denoted by $DM_{h,\ell}$ 
\end{definition}

We show that $DM_{h,\ell}$ satisfies the conditions of \Cref{thm:VSCmultiple}. This is done in the next three propositions.

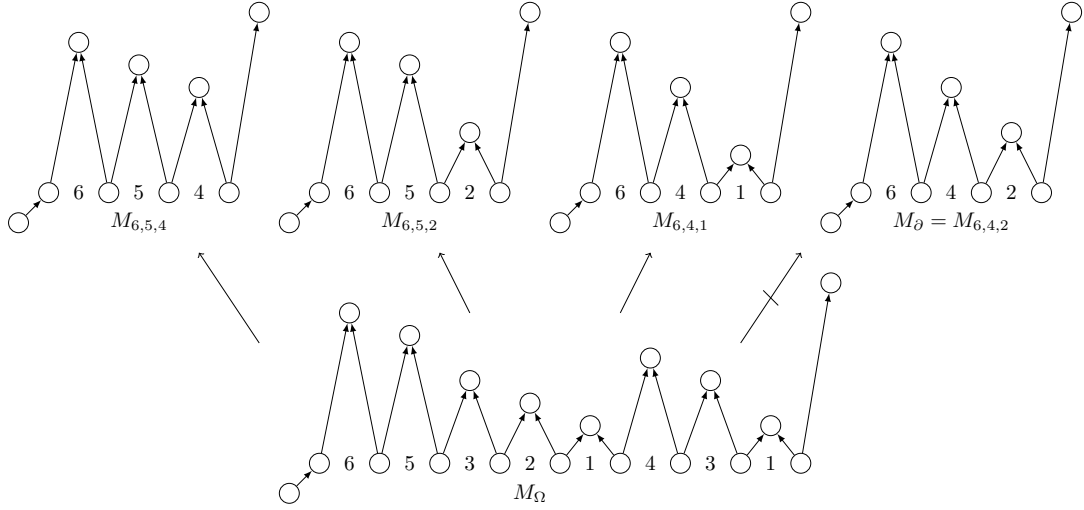
\begin{figure}[htb] 
 	\centering
 	\scalebox{.8}{\begin{tikzpicture}
	\begin{pgfonlayer}{nodelayer}
		\node [style=black empty] (0) at (-15, 7) {};
		\node [style=black empty] (1) at (-14, 8) {};
		\node [style=black empty] (3) at (-12, 8) {};
		\node [style=black empty] (4) at (-13, 13) {};
		\node [style=black empty] (5) at (-11, 12.25) {};
		\node [style=black empty] (6) at (-10, 8) {};
		\node [style=black empty] (7) at (-9, 11.5) {};
		\node [style=black empty] (8) at (-8, 8) {};
		\node [style=black empty] (9) at (-7, 14) {};
		\node [style=none] (10) at (-13, 8) {6};
		\node [style=none] (11) at (-11, 8) {5};
		\node [style=none] (13) at (-11, 7) {$M_{6,5,4}$};
		\node [style=black empty] (14) at (12, 7) {};
		\node [style=black empty] (15) at (13, 8) {};
		\node [style=black empty] (16) at (15, 8) {};
		\node [style=black empty] (17) at (14, 13) {};
		\node [style=black empty] (18) at (16, 11.5) {};
		\node [style=black empty] (19) at (17, 8) {};
		\node [style=black empty] (20) at (18, 10) {};
		\node [style=black empty] (21) at (19, 8) {};
		\node [style=black empty] (22) at (20, 14) {};
		\node [style=none] (23) at (14, 8) {6};
		\node [style=none] (24) at (16, 8) {4};
		\node [style=none] (25) at (18, 8) {2};
		\node [style=none] (26) at (16, 7) {$M_{\partial} = M_{6,4,2}$};
		\node [style=black empty] (27) at (-6, 7) {};
		\node [style=black empty] (28) at (-5, 8) {};
		\node [style=black empty] (29) at (-3, 8) {};
		\node [style=black empty] (30) at (-4, 13) {};
		\node [style=black empty] (31) at (-2, 12.25) {};
		\node [style=black empty] (32) at (-1, 8) {};
		\node [style=black empty] (33) at (0, 10) {};
		\node [style=black empty] (34) at (1, 8) {};
		\node [style=black empty] (35) at (2, 14) {};
		\node [style=none] (36) at (-4, 8) {6};
		\node [style=none] (37) at (-2, 8) {5};
		\node [style=none] (38) at (0, 8) {2};
		\node [style=none] (39) at (-2, 7) {$M_{6,5,2}$};
		\node [style=black empty] (40) at (3, 7) {};
		\node [style=black empty] (41) at (4, 8) {};
		\node [style=black empty] (42) at (6, 8) {};
		\node [style=black empty] (43) at (5, 13) {};
		\node [style=black empty] (44) at (7, 11.5) {};
		\node [style=black empty] (45) at (8, 8) {};
		\node [style=black empty] (46) at (9, 9.25) {};
		\node [style=black empty] (47) at (10, 8) {};
		\node [style=black empty] (48) at (11, 14) {};
		\node [style=none] (49) at (5, 8) {6};
		\node [style=none] (50) at (7, 8) {4};
		\node [style=none] (51) at (9, 8) {1};
		\node [style=none] (52) at (7, 7) {$M_{6,4,1}$};
		\node [style=black empty] (53) at (-6, -2) {};
		\node [style=black empty] (54) at (-5, -1) {};
		\node [style=black empty] (55) at (-3, -1) {};
		\node [style=black empty] (56) at (-4, 4) {};
		\node [style=black empty] (57) at (-2, 3.25) {};
		\node [style=black empty] (58) at (-1, -1) {};
		\node [style=black empty] (59) at (0, 1.75) {};
		\node [style=black empty] (60) at (1, -1) {};
		\node [style=none] (62) at (-4, -1) {6};
		\node [style=none] (63) at (-2, -1) {5};
		\node [style=none] (64) at (0, -1) {3};
		\node [style=none] (65) at (2, -2) {$M_{\Omega}$};
		\node [style=none] (66) at (2, -1) {2};
		\node [style=none] (67) at (4, -1) {1};
		\node [style=none] (68) at (6, -1) {4};
		\node [style=black empty] (80) at (3, -1) {};
		\node [style=black empty] (81) at (5, -1) {};
		\node [style=black empty] (82) at (7, -1) {};
		\node [style=black empty] (95) at (2, 1) {};
		\node [style=black empty] (96) at (4, 0.25) {};
		\node [style=black empty] (97) at (6, 2.5) {};
		\node [style=black empty] (98) at (8, 1.75) {};
		\node [style=none] (99) at (0, 4) {};
		\node [style=none] (100) at (-1, 6) {};
		\node [style=none] (101) at (5, 4) {};
		\node [style=none] (102) at (6, 6) {};
		\node [style=none] (103) at (-7, 3) {};
		\node [style=none] (104) at (-9, 6) {};
		\node [style=none] (105) at (9, 3) {};
		\node [style=none] (106) at (11, 6) {};
		\node [style=none] (107) at (9.75, 4.75) {};
		\node [style=none] (108) at (10.25, 4.25) {};
		\node [style=none] (110) at (-4, 8) {};
		\node [style=none] (111) at (-9, 8) {4};
		\node [style=none] (112) at (10, -1) {1};
		\node [style=black empty] (113) at (9, -1) {};
		\node [style=black empty] (114) at (11, -1) {};
		\node [style=black empty] (115) at (10, 0.25) {};
		\node [style=black empty] (116) at (12, 5) {};
		\node [style=none] (117) at (8, -1) {3};
	\end{pgfonlayer}
	\begin{pgfonlayer}{edgelayer}
		\draw [style=Arc] (0) to (1);
		\draw [style=Arc] (1) to (4);
		\draw [style=Arc] (3) to (4);
		\draw [style=Arc] (3) to (5);
		\draw [style=Arc] (6) to (5);
		\draw [style=Arc] (6) to (7);
		\draw [style=Arc] (8) to (7);
		\draw [style=Arc] (8) to (9);
		\draw [style=Arc] (14) to (15);
		\draw [style=Arc] (15) to (17);
		\draw [style=Arc] (16) to (17);
		\draw [style=Arc] (16) to (18);
		\draw [style=Arc] (19) to (18);
		\draw [style=Arc] (19) to (20);
		\draw [style=Arc] (21) to (20);
		\draw [style=Arc] (21) to (22);
		\draw [style=Arc] (27) to (28);
		\draw [style=Arc] (28) to (30);
		\draw [style=Arc] (29) to (30);
		\draw [style=Arc] (29) to (31);
		\draw [style=Arc] (32) to (31);
		\draw [style=Arc] (32) to (33);
		\draw [style=Arc] (34) to (33);
		\draw [style=Arc] (34) to (35);
		\draw [style=Arc] (40) to (41);
		\draw [style=Arc] (41) to (43);
		\draw [style=Arc] (42) to (43);
		\draw [style=Arc] (42) to (44);
		\draw [style=Arc] (45) to (44);
		\draw [style=Arc] (45) to (46);
		\draw [style=Arc] (47) to (46);
		\draw [style=Arc] (47) to (48);
		\draw [style=Arc] (53) to (54);
		\draw [style=Arc] (54) to (56);
		\draw [style=Arc] (55) to (56);
		\draw [style=Arc] (55) to (57);
		\draw [style=Arc] (58) to (57);
		\draw [style=Arc] (58) to (59);
		\draw [style=Arc] (60) to (59);
		\draw [style=Arc] (60) to (95);
		\draw [style=Arc] (80) to (95);
		\draw [style=Arc] (80) to (96);
		\draw [style=Arc] (81) to (96);
		\draw [style=Arc] (81) to (97);
		\draw [style=Arc] (82) to (97);
		\draw [style=Arc] (82) to (98);
		\draw [style=automorphism] (103.center) to (104.center);
		\draw [style=automorphism] (99.center) to (100.center);
		\draw [style=automorphism] (101.center) to (102.center);
		\draw [style=automorphism] (105.center) to (106.center);
		\draw [style=Edge] (108.center) to (107.center);
		\draw [style=Arc] (113) to (115);
		\draw [style=Arc] (114) to (115);
		\draw [style=Arc] (114) to (116);
		\draw [style=Arc] (113) to (98);
	\end{pgfonlayer}
\end{tikzpicture}}
 	\caption{Example of decreasing mountains, and the sequence $\Omega$ defined in the proof of \Cref{mountain_indep}} \label{fig_omega_mountain}
\end{figure}

\begin{proposition}
	For every member $M_{\partial}$ of $DM_{h,l}$ we have: $\prod\limits_{M_{\delta} \neq M_{\partial}}{M_{\delta}} \ra \cone M_{\partial}$
\end{proposition}
\begin{proof}
	In fact each $M_{\delta}$ maps to a digon, therefore $\prod\limits_{M_{\delta} \neq M_{\partial}}{M_{\delta}}$ maps to a digon.
\end{proof}

\begin{proposition}
	For every member $M_{\partial}$ of $DM_{h,l}$ we have: $\prod\limits_{M_{\delta}\in DM_{h,l}}{M_{\delta}}\rra M_{\partial}$
\end{proposition}
\begin{proof}
	As seen in \Cref{prop:tensorPathProduct}, $\prod\limits_{M_{\delta} \neq M_{\partial}}{M_{\delta}}$ contains an $h$-bounded path. Hence, any mapping of it to $ M_{\partial}$ must be surjective.
\end{proof}

\begin{proposition} \label{mountain_indep}
	For every member $M_{\partial}$ of $DM_{h,l}$ we have: $\prod\limits_{M_{\delta} \neq M_{\partial}}{M_{\delta}} \not \ra  M_{\partial}$
\end{proposition}
\begin{proof}
	Fix $\partial$ a decreasing sequence starting at $\partial_1 = h-2$.
	We shall construct a sequence (not necessarily decreasing) $\Omega$ such that $M_\Omega \not \ra M_\partial$ but for any $\delta \neq \partial$, $M_\Omega \ra M_\delta$.
	
	Let $[m]\setminus \set{k}$ denote the (decreasing) sequence $m,m-1,\dots,k+1,k-1,\dots,1$.
	We define $\Omega$ as concatenation of many sequences as follows:
	$$\Omega = \partial_1 \cdot \left ([\partial_1-1]\setminus \set{\partial_2}\right ) \cdot \partial_2 \cdot  \left ([\partial_2-1]\setminus \set{\partial_3}\right )  \ldots  \partial_{\ell - 1}  \cdot \left ([\partial_{\ell-1}-1]\setminus \set{\partial_l}\right ) .$$
	
	First, we show that $\Omega$ doesn't contain $\partial$ as a subsequence. This which would imply by \Cref{prop:list_hom} that $M_\Omega \not \ra M_\partial$.
	Toward a contradiction, assume that $\partial$ is a subsequence of $\Omega$, then $\partial_2$ must appear in the subsequence, it cannot appear inside $[\partial_1-1]\setminus \set{\partial_2}$, therefore, it must appear inside the part: $\partial_2 \cdot  \left ([\partial_2-1]\setminus \set{\partial_3}\right )  \ldots  \partial_{\ell - 1}  \cdot \left ([\partial_{\ell-1}-1]\setminus \set{\partial_l}\right )$.
	As $\partial_3$ must appear after $\partial_2$, and cannot appear in $[\partial_2-1]\setminus \set{\partial_3}$, it must appear in $\partial_3 \cdot  \left ([\partial_3-1]\setminus \set{\partial_4}\right )  \ldots  \partial_{\ell - 1}  \cdot \left ([\partial_{\ell-1}-1]\setminus \set{\partial_l}\right )$. Continuing this process, there is no place for $\partial_\ell$ to appear.
	\vspace{5pt}
	
	Finally, we show that given a $\delta \neq \partial$, we have $M_\Omega \ra M_\delta$.
	Since $\delta \neq \partial$, there exists a first index $t$ such that $\partial_t \notin \delta$, and we can decompose $\delta$ in the following way: 
	$$\delta = \partial_1\cdot\delta|_{]\partial_2,\partial_1[}\cdot \partial_2\cdot\delta|_{]\partial_3,\partial_2[} \dots \partial_{t-1} \cdot \delta|_{]0,\partial_{t-1}[}.$$
	Where $\delta|_{]i,j[}$ is the (possibly empty) subsequence of $\delta$ of the values in $[i,j]$.
	This decomposition corresponds to a natural subsequence of $\Omega$, as $\delta|_{]\partial_i,\partial_{i+1}[} $ is a subsequence of $[\partial_i] \setminus \set{\partial_{i+1}}$, and $\delta|_{]0,\partial_{t-1}[}$ is a subsequence of $[\partial_{t-1}] \setminus \set{\partial_{t}}$.
	
	With this presentation of $\delta$ as a subsequence of $\Omega$, it remains to verify the last two conditions of $\Cref{def:list_hom}$, in order to verify that $\Omega$ is homomorphic to $\delta$ that is to say, $M_\Omega \ra M_\delta$.
	The second condition of the definition is trivial because $\delta_1 = \Omega_1 = \partial_1 = h+2$.
	By the nature of our choices, the subsequence of $\Omega$ in between $\delta_i$ and $\delta_{i+1}$ is bounded by $\delta_i$
\end{proof}

\begin{theorem}\label{thm:M-partial}
	For any $h>l+2$, $\prod\limits_{M_\partial \in DM_{h,\ell}} \cone M_\partial$ is a core.
\end{theorem}

There is a one-to-one correspondance between elements of $DM_{h,\ell}$ and $\ell$-subsets of $[h-2]$. In other words, $|DM_{h,\ell}|=\binom{h-2}{\ell}$. Taking $\ell = k$ and $h=2k+2$, we have a set of $\binom{2k}{k}$ digraphs, the largest of which has $3k^2+3k+2$ vertices and the smallest of which has $k^2+5k+2$ vertices. Therefore:
$$ \left |\core{ \textstyle \prod\limits_{M_\partial \in DM_{2k+2,k}} \cone M_\partial} \right |  \geq (k^2+5k+2)^{\binom{2k}{k}}$$

\begin{proposition}
	There exists a family of digraphs, each of order at most $n$, whose product has a core of order $\Omega\left (n^{\frac{2^{^{\sqrt{n}}}}{\sqrt{n}}}\right )$
\end{proposition}

\section{From digraphs to graphs}\label{sec:Digraph-Graph}

In this final section we show that the family of digraphs we built whose product was a core, can be transformed into a family of graphs having the same property.
The idea is to replace every oriented edge with a non-oriented gadget where the structure of the gadget captures the orientation. The gadget we use is the graph $K_2\Bowtie C_5$. Given a digraph $D$ we build a graph $G[D]$ by replacing each edge with a (vertex disjoint) copy of $K_2\Bowtie C_5$ where the head is identified with a vertex from the $K_2$ part and the tail is replaced with a vertex from $C_5$-part. See \Cref{fig:wheel} for an example. We show that the following holds.

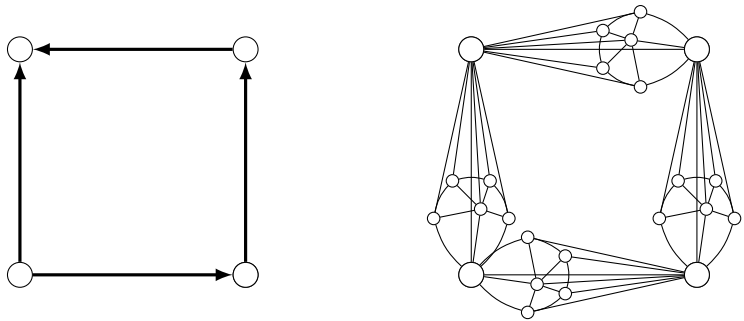
\begin{figure}[htb]
	\centering
	\scalebox{1}{\begin{tikzpicture}
	\begin{pgfonlayer}{nodelayer}
		\node [style=small black empty] (0) at (6.5, -2.5) {};
		\node [style=small black empty] (1) at (5.5, -2) {};
		\node [style=black empty] (2) at (4, -3) {};
		\node [style=small black empty] (3) at (6.5, -3.5) {};
		\node [style=small black empty] (4) at (5.5, -4) {};
		\node [style=black empty] (5) at (10, -3) {};
		\node [style=black empty] (6) at (-8, -3) {};
		\node [style=black empty] (7) at (-2, -3) {};
		\node [style=black empty] (8) at (-2, -3) {};
		\node [style=black empty] (9) at (-2, 3) {};
		\node [style=small black empty] (10) at (9.5, -0.5) {};
		\node [style=small black empty] (11) at (9, -1.5) {};
		\node [style=black empty] (12) at (10, -3) {};
		\node [style=small black empty] (13) at (10.5, -0.5) {};
		\node [style=small black empty] (14) at (11, -1.5) {};
		\node [style=black empty] (15) at (10, 3) {};
		\node [style=black empty] (22) at (-8, 3) {};
		\node [style=small black empty] (23) at (3.5, -0.5) {};
		\node [style=small black empty] (24) at (3, -1.5) {};
		\node [style=black empty] (25) at (4, -3) {};
		\node [style=small black empty] (26) at (4.5, -0.5) {};
		\node [style=small black empty] (27) at (5, -1.5) {};
		\node [style=black empty] (28) at (4, 3) {};
		\node [style=small black empty] (29) at (7.5, 3.5) {};
		\node [style=small black empty] (30) at (8.5, 4) {};
		\node [style=black empty] (31) at (10, 3) {};
		\node [style=small black empty] (32) at (7.5, 2.5) {};
		\node [style=small black empty] (33) at (8.5, 2) {};
		\node [style=black empty] (34) at (4, 3) {};
		\node [style=small black empty] (35) at (8.25, 3.25) {};
		\node [style=small black empty] (36) at (4.25, -1.25) {};
		\node [style=small black empty] (37) at (5.75, -3.25) {};
		\node [style=small black empty] (38) at (10.25, -1.25) {};
	\end{pgfonlayer}
	\begin{pgfonlayer}{edgelayer}
		\draw [style=Edge] (1) to (5);
		\draw [style=Edge] (5) to (2);
		\draw [style=Edge] (5) to (4);
		\draw [style=Edge] (5) to (3);
		\draw [style=Edge] (5) to (0);
		\draw [style=Edge, bend left=15, looseness=0.75] (1) to (0);
		\draw [style=Edge, bend left=15] (0) to (3);
		\draw [style=Edge, bend left=15] (3) to (4);
		\draw [style=Edge, bend left=15] (4) to (2);
		\draw [style=Edge, bend left=15] (2) to (1);
		\draw [style=big arc] (6) to (7);
		\draw [style=big arc] (8) to (9);
		\draw [style=Edge] (11) to (15);
		\draw [style=Edge] (15) to (12);
		\draw [style=Edge] (15) to (14);
		\draw [style=Edge] (15) to (13);
		\draw [style=Edge] (15) to (10);
		\draw [style=Edge, bend left=15] (11) to (10);
		\draw [style=Edge, bend left=15] (10) to (13);
		\draw [style=Edge, bend left=15] (13) to (14);
		\draw [style=Edge, bend left=15] (14) to (12);
		\draw [style=Edge, bend left=15] (12) to (11);
		\draw [style=big arc] (9) to (22);
		\draw [style=big arc] (6) to (22);
		\draw [style=Edge] (24) to (28);
		\draw [style=Edge] (28) to (25);
		\draw [style=Edge] (28) to (27);
		\draw [style=Edge] (28) to (26);
		\draw [style=Edge] (28) to (23);
		\draw [style=Edge, bend left=15] (24) to (23);
		\draw [style=Edge, bend left=15] (23) to (26);
		\draw [style=Edge, bend left=15] (26) to (27);
		\draw [style=Edge, bend left=15] (27) to (25);
		\draw [style=Edge, bend left=15] (25) to (24);
		\draw [style=Edge] (30) to (34);
		\draw [style=Edge] (34) to (31);
		\draw [style=Edge] (34) to (33);
		\draw [style=Edge] (34) to (32);
		\draw [style=Edge] (34) to (29);
		\draw [style=Edge, bend right=15] (30) to (29);
		\draw [style=Edge, bend right=15] (29) to (32);
		\draw [style=Edge, bend right=15] (32) to (33);
		\draw [style=Edge, bend right=15] (33) to (31);
		\draw [style=Edge, bend right=15] (31) to (30);
		\draw [style=Edge] (38) to (11);
		\draw [style=Edge] (38) to (10);
		\draw [style=Edge] (38) to (13);
		\draw [style=Edge] (38) to (14);
		\draw [style=Edge] (38) to (12);
		\draw [style=Edge] (35) to (30);
		\draw [style=Edge] (35) to (29);
		\draw [style=Edge] (35) to (32);
		\draw [style=Edge] (35) to (33);
		\draw [style=Edge] (35) to (31);
		\draw [style=Edge] (36) to (27);
		\draw [style=Edge] (36) to (26);
		\draw [style=Edge] (36) to (23);
		\draw [style=Edge] (36) to (24);
		\draw [style=Edge] (36) to (25);
		\draw [style=Edge] (37) to (1);
		\draw [style=Edge] (37) to (25);
		\draw [style=Edge] (37) to (4);
		\draw [style=Edge] (37) to (3);
		\draw [style=Edge] (37) to (0);
		\draw [style=Edge] (36) to (34);
		\draw [style=Edge] (37) to (12);
		\draw [style=Edge] (38) to (31);
		\draw [style=Edge] (35) to (34);
	\end{pgfonlayer}
\end{tikzpicture}}
		\caption{An example of $G[D]$}
		\label{fig:wheel}		
\end{figure}

\begin{theorem}\label{thm:D1*D2}
	Given two digraphs $D_1$ and $D_2$, if the underlying graph of $D_2$ is 4-colorable, then $D_1\ra D_2$ if and only if $G[D_1] \ra G[D_2]$. Furthermore, restriction of any homomorphism $\psi: G[D_1] \ra G[D_2]$ to the vertices of $D_1$ is a homomorphism of $D_1$ to $D_2$.
\end{theorem}

\begin{proof}
	We first observe that there are three possible homomorphic images of $K_2\Bowtie C_5$: the graph itself, one obtained by identifying two nonadjacent vertices of $C_5$ (which is isomorphic to $K_6-E(P_3)$), and $K_5$. 
	
	Let $\phi$ be a homomorphism of $G[D_1]$ to $G[D_2]$. For an arc $a_1= u_1 \tikzRarrow v_1$ of $D_1$ we claim that the copy of the gadget corresponding to $a_1$ in $G[D_1]$ is mapped to a copy built on an arc $a_2=u_2 \tikzRarrow v_2$ of $D_2$. This would complete the proof of the theorem because such a mapping must map $K_2$ and $C_5$ to the corresponding parts. 
	
	To prove the claim of the theorem, let $H_1$ be the copy of the gadget built on the arc $a_1$. As the underlying graph of $D_2$ is 4-colorable, $\phi(H_1)$ must contain a vertex $x$ not in $V(D_2)$. Suppose this vertex is in the copy $H_2$ of the gadget corresponding to an arc $a_2$ of $D_2$. As $H_2$ contains no $K_5$, if $\phi(H_1)$ is a $K_5$ or $K_6-E(P_3)$, then at least one of its vertices, say $y$, is not in $H_2$. Then the vertices of $\phi(V(H_1))\cap \{u_2, v_2\}$ (a subset of vertices of the arc $a_2$) separate the two vertices, $x$ and $y$ of $\phi(H_1)$, but both $K_5$ and $(K_6-E(P_3)$ are 3-connected. Thus $\phi(H_1)$ is isomorphic to $K_2\Bowtie C_5$, and can only be the copy built on arc $a_2$. Furthermore, the $K_2$ part of $H_1$ must maps to the $K_2$ part of $H_2$ which proves the second part of the claim.	
\end{proof}

As the diagonal entries of $G\times G$ induce an isomorphic copy of $G$, we have the following observation.

\begin{observation}
	Given digraphs $D_1$ and $D_2$ the graph $G[D_1\times D_2]$ is a subgraph of $G[D_1]\times G[D_2]$.
\end{observation}

With a repeated application of this observation, we concluded that for any (finite) family $\mathcal{D}$ of digraphs the graph $ G[\footnotesize  \prod\limits_{D \in \mathcal{D}} D]$ is a subgraph of $ {\prod\limits_{D \in \mathcal{D}} G[D]}$. As a final step we show that under certain conditions, the core of $ {\prod\limits_{D \in \mathcal{D}} G[D]}$ is at least as large as the core of  $ G[\footnotesize  \prod\limits_{D \in \mathcal{D}} D]$.

\begin{theorem}
If $\mathcal{D}$ is a (finite) family of digraphs each of whose underlying graph is 4-colorable, then $$ \left| G\left [\core{\textstyle\prod\limits_{D \in \mathcal{D}} D} \right ] \right| \leq \left|\core{\textstyle\prod\limits_{D \in \mathcal{D}} G[D]}\right|.$$	
\end{theorem}

\begin{proof}
	Let $\psi$ be a mapping of $\prod\limits_{D \in \mathcal{D}}G[D]$ to (one of) its core. Let $\psi_i$ be the projection of $\core{\textstyle\prod\limits_{D \in \mathcal{D}} G[D]}$ to $G[D_i]$. Since $G[\prod\limits_{D \in \mathcal{D}} D]$ is a subgraph of $\prod\limits_{D \in \mathcal{D}}G[D]$, the mapping $\psi_i$ can be reduced to a mapping $\psi_i'$ of $G[\prod\limits_{D \in \mathcal{D}}D]$ to $G[D_i]$.
	
	Since each $D_i$ is 4-colorable, the product $\prod\limits_{D \in \mathcal{D}}{D}$ is also 4-colorable. Thus, by \Cref{thm:D1*D2}, $\psi'_i$ induces a mapping of $\prod\limits_{D \in \mathcal{D}}D$ to $D_i$. Hence, $\psi$ induces a mapping of $\prod\limits_{D \in \mathcal{D}}D$ to itself which should contain a copy, say $H$, of its core. Applying \Cref{thm:D1*D2} again we conclude that the image of $H$ must contain a copy of $G[H]$.
\end{proof}

As a corollary, when we take the family $DM_{h, l}$ consisting of the cones of decreasing mountains, and build $G[\cone{P}]$ for each $P\in  DM_{h, l}$, their product contains $G[H]$ where $H$ is the digraph $\prod\limits_{P\in D_{h,l}}{\cone P}$. Observe that if $P$ has $n$ vertices, then $\cone P$ has $n+1$ vertices and $3n-1$ edges. Hence $G[\cone{P}]$ has $n+1+5(3n-1)=16n-4$ vertices. Bounding the number of vertices of graphs in $DM_{2k, k}$ below and above by $k^2+5k+2$ and  $3k^2+3k+2$ we have the following conclusion based on the core of the product  $\prod\limits_{P\in D_{2k,k}}{\cone P}$.

\begin{corollary}
	There is a family of $\binom{2k}{k}$ graphs each of which has its number of vertices between $16k^2+80k+12$ and $48k^2+48k+12$  where the core of the product has more than $4(16k^2+80k+2)^{\binom{2k}{k}}$.
\end{corollary}

\section*{Acknowledgement}
This work has received support under the program ``Investissement d'Avenir" launched by the French Government and implemented by ANR, with the reference ``ANR‐18‐IdEx‐0001" as part of its program ``Emergence".\\

We would also like to thank Leonid Libkin for introducing the problem to us. 

\bibliographystyle{plain}
\bibliography{Bibtex2023}

\end{document}